\documentclass[12pt,reqno]{amsart}%
\usepackage{graphicx}
\usepackage{amsmath}
\usepackage{amsfonts}
\usepackage{amssymb}
\usepackage{mathtools,amsthm}
\usepackage{enumerate}
\usepackage{color}
\usepackage{cite}
\usepackage{url}
\usepackage[margin=1in]{geometry}%
\usepackage{tikz-cd}
\usepackage{mathabx}
\usepackage{boxedminipage}
\usepackage{algorithmic}
\usepackage{hyperref}
\usepackage[normalem]{ulem}
\usepackage{float}
\usepackage{graphicx}
\usepackage{acronym}
\providecommand{\U}[1]{\protect\rule{.1in}{.1in}}

\newtheorem{theorem}{Theorem}[section]
\newtheorem{proposition}[theorem]{Proposition}
\newtheorem{lemma}[theorem]{Lemma}
\newtheorem{corollary}[theorem]{Corollary}
\newtheorem{definition}[theorem]{Definition}

\newtheorem{algorithm}[theorem]{Algorithm}

\theoremstyle{remark}

\let\O\undefined

\DeclareMathOperator{\O}{O}

\DeclareMathOperator{\tr}{tr}

\DeclareMathOperator{\V}{V}

\newcommand{\tp}{{\scriptscriptstyle\mathsf{T}}}

\begin{document}
\title[Low rank orthogonal tensor approximation]{Linear Convergence of an Alternating Polar Decomposition Method for Low Rank Orthogonal Tensor Approximations}

\author{Shenglong Hu}
\address{Department of Mathematics, School of Science, Hangzhou Dianzi University, Hangzhou 310018, China.}
\email{shenglonghu@hdu.edu.cn}

\author{Ke Ye}
\address{KLMM, Academy of Mathematics and Systems Science, Chinese Academy of Sciences, Beijing 100190, China.}
\email{keyk@amss.ac.cn}

\begin{abstract}
Low rank orthogonal tensor approximation (LROTA) is an important problem in tensor computations and their applications. A classical and widely used algorithm is the alternating polar decomposition method (APD).
In this paper, an improved version iAPD of the classical APD is proposed. For the first time, all of the following four fundamental properties are established for iAPD: (i) the algorithm converges globally and the whole sequence converges to a KKT point without any assumption; (ii) it exhibits an overall sublinear convergence with an explicit rate which is sharper than the usual $O(1/k)$ for first order methods in optimization; (iii) more importantly, it converges $R$-linearly for a generic tensor without any assumption; (iv) for almost all LROTA problems, iAPD reduces to APD after finitely many iterations if it converges to a local minimizer.
\end{abstract}


\maketitle



\section{Introduction}\label{sec:intro}
As higher order {generalizations} of matrices, tensors (a.k.a. hypermatrices or multi-way arrays) are ubiquitous and inevitable in {mathematical modeling} and scientific computing \cite{L-12,H-12,QL-17,CJ-10,M-87}. Among numerous tensor problems studied in recent years, tensor approximation and its related topics have been becoming the {main} focus, see \cite{KB-09,L-12,C-14} and references therein. Applications of tensor approximation are diverse and broad, including signal processing \cite{CJ-10}, computational complexity \cite{L-12}, pattern recognition \cite{AGHKT-14}, principal component analysis \cite{C-94}, etc. We refer interested readers to surveys \cite{KB-09,GKT-13,C-14,L-13} and books \cite{L-12,H-12,QL-17} for more details.

Singular value decomposition (SVD) of matrices is both a theoretical foundation and a computational workhorse for linear algebra, with applications spreading throughout scientific {computing} and engineering \cite{GV-13}. SVD of a given matrix is {a} \textit{rank-one orthogonal decomposition} of {the} matrix \cite{HJ-85,GV-13}, and a truncated SVD according to the {non-increasing} {singular values} is {a} \textit{low rank orthogonal approximation} of that matrix by the {well-known} Echart-Young theorem \cite{EY-36}. While a higher order tensor cannot be diagonalized by orthogonal matrices in general \cite{L-12}, there are several {generalizations} of SVD from matrices {to} tensors, such as higher order SVD \cite{DDV-00}, \textit{orthogonally decomposable tensor} decompositions and approximations and their variants, see \cite{K-01,CS-09,WTY-15,MHG-15,SDCJD-12,F-92} and references therein. In this paper, we focus on the low rank orthogonal tensor approximation (abbreviated as LROTA) problem. It is a low rank tensor approximation problem with all the factor matrices being orthonormal \cite{C-14,KB-09,DL-08}. This problem is of crucial importance in applications, such as blind source separation in signal processing and statistics \cite{CJ-10,C-94,C-14,M-87}.

In the literature, several numerical methods have been proposed to solve this problem, such as Jacobi-type methods \cite{C-94}, for which the tensor considered usually has a symmetric structure. {Interested} readers are referred to \cite{LUC-18,LUC-19,ULC-19,LUC-20,MV-08,IAV-13}. A more general problem is studied, {where} some factor matrices {are} orthonormal and the rest of them {are} unconstrained. We denote these low rank orthogonal tensor approximation problems {by} LROTA-$s$ with $s$ the number of orthonormal factor matrices. For simplicity, LROTA {denotes the problem where all the factor matrices are orthonormal}. For LROTA-$s$, a {commonly adopted algorithmic framework is the \textit{alternating minimization method (AMM)} \cite{B-99}, under which the alternating polar decomposition method (APD) is proposed and widely employed \cite{CS-09,WTY-15}.}
Under a regularity condition that all matrices in certain iterative sequence are of full rank, it is {proved} that every converging subsequence generated by this method for LROTA converges to a stationary point of the objective function by Chen and Saad \cite{CS-09} in 2009. In 2012, Uschmajew established a local convergence result under {some} appropriate assumptions \cite{U-12}. In 2015, Wang, Chu and Yu proposed {an} AMM with a modified polar decomposition for LROTA-$1$, and established the global convergence without any further assumption for a generic tensor \cite{WTY-15}. In 2019, Guan and Chu \cite{GC-19} established the global convergence for LROTA-$s$ with general $s$ under a similar regularity condition as \cite{CS-09}. Very recently, Yang proposed an epsilon-alternating least square method for solve the problem LROTA-$s$ with general $s$ and established its global convergence without any assumption \cite{Y-19} \footnote{Yang's paper is posted during our final preparation of this paper. We can see that we both employ the proximal technique. A difference is that the proximal correction in our algorithm is adaptive, and a theoretical investigation is also given (cf.\ Section~\ref{sec:discussion}) for the execution of the proximal correction.}. {On the other hand, the special case of rank-one tensor approximation} was {systematically} studied since {the work of} De Lathauwer, De Moor, and Vandewalle \cite{DDV-00}. {A higher order power method, which is essentially an application of AMM, was proposed and global} convergence results were established, see \cite{M-13,WC-14,U-15}. {Moreover, the convergence rate} was also estimated in \cite{ZG-01,U-12,HL-18}.


Motivated by the {development} of the convergence analysis of the rank-one case and the general low rank case, a fundamental question is: \textit{is there an algorithm for LROTA such that all the favorable convergence properties in the rank-one case {also} hold for the general low rank case}? The answer is affirmative. Given this is true, one hopes that this algorithm should be as close as possible to the widely used classical APD, so that several questions raised in the literature can be addressed \cite{WTY-15,GC-19,CS-09}. {In this paper, we provide an affirmative answer to the question.
} Listed below are main contributions of this paper:
\begin{enumerate}
\item we propose an improved version iAPD for the alternating polar decomposition (APD) method given in \cite{GC-19} for solving LROTA, and show its global convergence without any assumption;
\item we establish an overall sublinear convergence of iAPD and present an \textit{explicit eventual convergence rate} in terms of the dimension and the order of the underlying tensor. The derived convergence rate is sub-optimal in the sense that it is sharper than the usual convergence rate $O(1/k)$ established for first order methods in the literature \cite{Beck-book};
\item we prove that iAPD is linear convergent for a generic tensor without any other assumption; 
\item we also show that for almost all LROTA problems, iAPD reduces to APD after finitely many iterations if it converges to a local minimizer. In particular, this relaxes the requirement for {each iterative matrix being of full rank} in the literature, such as \cite{GC-19,CS-09}, to a simple requirement on the limit point.
\end{enumerate}
{Other than these, we also prove} that every KKT point of LROTA is nondegenerate for a generic tensor, which {might} be of independent interests.

The {rest} of this paper {is} organized as follows: {preliminaries on multilinear algebra, differential geometry and optimization theory that will be encountered repeatedly in the sequel are included in Section~\ref{sec:preliminary}.} In particular, the LROTA problem is {stated} in Section~\ref{sec:lowrank}; Section~\ref{sec:manifold} is devoted to the analysis of the manifold structures of the set of orthogonally decomposable tensors. In this section, the connection between KKT points of LROTA and critical points of the projection function on manifolds is established. It is shown that every KKT point of LROTA for a generic tensor is nondegenerate; the {new} algorithm iAPD is proposed in Section~\ref{sec:APPDT} and detailed convergence analysis for this algorithm is given. The overall sub-optimal sublinear convergence and generic linear convergence are also proved; Section~\ref{sec:discussion} {proves} that for almost all LROTA problems, iAPD reduces to APD after finitely many iterations if it converges to a local minimizer; some final remarks are given in Section~\ref{sec:final}; {to avoid distracting readers too much by technical details, lemmas are stated when they are needed and proofs
are provided in Appendix~\ref{sec:polar} and \ref{app:proofalgo}.}

\section{Preliminaries}\label{sec:preliminary}

\subsection{Tensors}\label{sec:notation}
In this {subsection}, we provide {a review of basic notions of tensors.} Given positive integers $k\geq 2$ and $n_1,\dots,n_k$, the tensor space consisting of real tensors of dimension $n_1\times\dots\times n_k$ is denoted {by} $\mathbb R^{n_1}\otimes\dots\otimes\mathbb R^{n_k}$. In this {vector space}, an inner product {and hence a norm} can be defined. The Hilbert-Schmidt inner product of two given tensors $\mathcal A, \mathcal B\in\mathbb R^{n_1}\otimes\dots\otimes\mathbb R^{n_k}$ is defined {by}
\[
\langle\mathcal A,\mathcal B\rangle:=\sum_{i_1=1}^{n_1}\dots\sum_{i_k=1}^{n_k}a_{i_1\dots i_k}b_{i_1\dots i_k}.
\]
The Hilbert-Schmidt norm $\|\mathcal A\|$ is then {given by} (cf.\ \cite{L-13})
\[
\|\mathcal A\|:=\sqrt{\langle\mathcal A,\mathcal A\rangle}.
\]
{
In particular, if $k=2$, then an element in $\mathbb{R}^{n_1} \otimes \mathbb{R}^{n_2}$ is simply an $n_1 \times n_2$ matrix $A$, whose Hilbert-Schmidt norm reduces to the Frobenius norm $\lVert A \rVert_F$.
}

Given a positive integer $r$ and $\lambda_1,\dots,\lambda_r\in\mathbb R$, {we denote by $\operatorname{diag}(\lambda_1,\dots,\lambda_r)$ the diagonal tensor in $\mathbb R^r\otimes\dots\otimes\mathbb R^r$ with the order being understood from the context}. To be more precise, we have
\[
\left( \operatorname{diag} (\lambda_1,\dots,\lambda_r)\right)_{i_1\dots i_k} =
\begin{cases}
\lambda_j,\quad &\text{if}~i_1 =\cdots = i_k = j\in\{1,\dots,r\}, \\
0 ,\quad &\text{otherwise}.
\end{cases}
\]
For a given positive integer $k$, we may regard the tensor $\operatorname{diag} (\lambda_1,\dots,\lambda_r)$ as the image of $(\lambda_1,\dots,\lambda_r)$ under the map $\operatorname{diag} : \mathbb R^r\to\otimes^k\mathbb R^r$ defined in an obvious way. We also define the map $\operatorname{Diag} : \otimes^k\mathbb R^r\to\mathbb R^r$ by taking the diagonal of a $k$th order tensor. By definition, $\operatorname{Diag}\circ\operatorname{diag} : \mathbb R^r\to\mathbb R^r$ is the identity map.


{We define a map $\tau : \mathbb R^{n_1}\times\dots\times\mathbb R^{n_k}\rightarrow\mathbb R^{n_1}\otimes\dots\otimes\mathbb R^{n_k}$ by
\begin{equation}\label{eq:segre}
\tau(\mathbf x):=\mathbf x_1\otimes\dots\otimes\mathbf x_k.
\end{equation}
where $\mathbf{x}$ is a \textit{block vector}
\[
\mathbf x:=(\mathbf x_1,\dots,\mathbf x_k)\in\mathbb R^{n_1}\times\dots\times\mathbb R^{n_k}\simeq\mathbb R^{n_1+\dots+n_k}\ \text{with }\mathbf x_i\in\mathbb R^{n_i}\ \text{for all }i=1,\dots,k.
\]
For each $i\in\{1,\dots,k\}$, we define a map $\tau_i : \mathbb R^{n_1}\times\dots\times\mathbb R^{n_k}\rightarrow\mathbb R^{n_1}\otimes\dots\otimes\mathbb R^{n_{i-1}}\otimes\mathbb R^{n_{i+1}}\otimes\dots\otimes\mathbb R^{n_k}$ by
\[
\tau_i(\mathbf x):=\mathbf x_1\otimes\dots\otimes\mathbf x_{i-1}\otimes\mathbf x_{i+1}\otimes\dots\otimes\mathbf x_k,\quad \mathbf{x} \in \mathbb R^{n_1}\times\dots\times\mathbb R^{n_k}.
\]
}

Given a tensor {$\mathcal A\in\mathbb R^{n_1}\otimes\dots\otimes\mathbb R^{n_k}$} and a block vector $\mathbf x$ as above, $\mathcal A\tau(\mathbf x)$ is defined by
\[
\mathcal A\tau(\mathbf x):=\langle\mathcal A,\tau(\mathbf x)\rangle
\]
and $\mathcal A\tau_i(\mathbf x)\in\mathbb R^{n_i}$ {is a vector defined implicitly by the relation}:
\[
\langle \mathcal A\tau_i(\mathbf x),\mathbf x_i\rangle=\mathcal A\tau(\mathbf x)
\]
for any block vector $\mathbf x$. Moreover, given $k$ matrices $B^{(i)}\in\mathbb R^{m_i\times n_i}$ for $i\in\{1,\dots,k\}$, the {\textit{matrix-tensor product} $(B^{(1)},\dots,B^{(k)})\cdot\mathcal A$ is a tensor} in $\mathbb R^{m_1}\otimes\dots\otimes\mathbb R^{m_k}$, defined entry-wisely as
\begin{equation}\label{eq:matrix-tensor}
\big[(B^{(1)},\dots,B^{(k)})\cdot\mathcal A\big]_{i_1\dots i_k}:=\sum_{j_1=1}^{n_1}\dots\sum_{j_k=1}^{n_k}b^{(1)}_{i_1j_1}\dots b^{(k)}_{i_kj_k}a_{j_1\dots j_k}
\end{equation}
for all $i_t\in\{1,\dots,m_t\}$ and $t\in\{1,\dots,k\}$.

\subsection{Stiefel {manifolds}}\label{sec:stiefel}
Let $m\le n$ be two positive integers and let $V(m,n)\subseteq \mathbb R^{n \times m}$ be the set of all $n \times m$ orthonormal matrices, i.e.,
\[
V(m,n):=\{U\in \mathbb R^{n \times m}\colon U^\tp U=I\},
\]
where $I$ is the $m\times m$ identity matrix. {Indeed, $V(m,n)$ admits a smooth manifold structure and is called \emph{the Stiefel manifold of orthonormal $m$-frames in $\mathbb{R}^n$}.} In particular, if $m = n$ then $V(n,n)$ simply reduces to the orthogonal group $\O (n)$.

For any $A\in V(m,n)$, the \emph{Fr\'echet normal cone} of $V(m,n)$ at $A$ is defined as (cf.\ \cite{RW-98})
\[
{\hat N_{V(m,n)}(A):=\{B\in\mathbb R^{n \times m}\mid \langle B,C-A\rangle\leq o(\|C-A\|_F)\ \text{for all }C\in V(m,n)\}.}
\]
Usually, we set $\hat N_{V(m,n)}(A) = \emptyset$ whenever $A\not\in V(m,n)$.
The \emph{(limiting) normal cone} $N_{V(n,m)}(A)$ of $V(n,m)$ at $A\in V(n,m)$ is defined by (cf.\ \cite{RW-98})
{\[
N_{V(m,n)}(A) \coloneqq \left\lbrace B \in \mathbb{R}^{n\times m}:
\begin{multlined}
A_k\in V(m,n),  \lim_{k\to \infty} A_k = A, \\
B_k\in \hat N_{V(m,n)}(A_k), \lim_{k\to \infty} B_k = B
\end{multlined}
\right\rbrace.
\]
}

It is easily seen from the definition that the normal cone $N_{V(m,n)}(A)$ is always closed.
The indicator function $\delta_{V(m,n)}$ of $V(m,n)$ is defined by
\[
\delta_{V(m,n)}(X):=\begin{cases}0&\text{if }X\in V(m,n),\\ +\infty &\text{otherwise}.\end{cases}
\]
Given a function $f : \mathbb R^n\rightarrow \mathbb R\cup\{\infty\}$, the \emph{regular subdifferential} of $f$ at $\mathbf x\in\mathbb R^n$ is defined as
\[
{\hat\partial f(\mathbf x):=\Bigg\{\mathbf v\in\mathbb R^n\colon\liminf_{\mathbf x\neq\mathbf y\rightarrow \mathbf x}\frac{f(\mathbf y)-f(\mathbf x)-\langle\mathbf v,\mathbf y-\mathbf x\rangle}{\|\mathbf y-\mathbf x\|}\geq 0\Bigg\}
}
\]
and the \emph{(limiting) subdifferential} of $f$ at $\mathbf x$ is defined as
\[
{\partial f(\mathbf x):=\Big\{\mathbf v\in\mathbb R^n\colon \exists \{\mathbf x^k\}\rightarrow \mathbf x\ \text{and }\{\mathbf v^k\}\rightarrow \mathbf v\ \text{satisfying }\mathbf v^k\in \hat\partial f(\mathbf x^k)\ \text{for all }k\Big\}.
}
\]
If $\mathbf 0\in\partial f(\mathbf x)$, then $\mathbf x$ is a \textit{critical point} of $f$.
An important fact about the normal cone $N_{V(m,n)}(A)$ and the subdifferential of the indicator function $\delta_{V(m,n)}$ of $V(m,n)$ at $A$ is (cf.\ \cite{RW-98})
\begin{equation}\label{eq:normal-sub}
\partial \delta_{V(m,n)}=N_{V(m,n)}.
\end{equation}
Note that $V(m,n)$ is a smooth manifold of dimension $mn-\frac{m(m+1)}{2}$.
It follows from \cite[Chapter~6.C]{RW-98} and \cite{EAT-98,AMS-08} that
\[
N_{V(m,n)}(A)=\hat N_{V(m,n)}(A)=\{AS\mid S\in\operatorname{S}^{m \times m}\},
\]
where $\operatorname{S}^{m\times m}\subseteq \mathbb R^{m\times m}$ is the subspace of $m \times m$ symmetric matrices.

Given a matrix $B\in\mathbb R^{n \times m}$, the projection of $B$ onto the normal cone of $V(m,n)$ at $A$ is
\[
\pi_{N_{V(m,n)}(A)}(B)=A(\frac{A^\tp B+B^\tp A}{2}).
\]

The tangent space $T_{V(m,n)}(A)$ of $V(m,n)$ at a point $A\in V(m,n)$ is the orthogonal complement to the normal cone.
Given a matrix $B\in\mathbb R^{n \times m}$, the projection of $B$ onto the tangent space of $V(m,n)$ at a point $A\in V(m,n)$ is given by
\begin{equation}\label{eq:tangent-proj}
\pi_{T_{V(m,n)}(A)}(B)=A\operatorname{skew}(A^\tp B)+(I-AA^\tp )B,
\end{equation}
where $\operatorname{skew}(C) \coloneqq \frac{C-C^\tp }{2}
$ is for a square matrix $C\in\mathbb R^{m \times m}$. A more explicit formula is given as
\begin{equation}\label{eq:tangent-form}
\pi_{T_{V(m,n)}(A)}(B)= (I-\frac{1}{2}AA^\tp )(B-AB^\tp A).
\end{equation}
\subsection{Orthogonally decomposable tensor}\label{codt}
A tensor $\mathcal A\in\mathbb R^{n_1}\otimes\dots\otimes\mathbb R^{n_k}$ is called \textit{orthogonally decomposable} (cf.\ \cite{ZG-01,K-01,F-92,K-02}) if there exist orthonormal matrices
\[
U^{(i)}= \begin{bmatrix}
\mathbf u^{(i)}_1 & \cdots & \mathbf u^{(i)}_r
\end{bmatrix} \in V(r,n_i),\ i=1,\dots,k
\]
and numbers $\lambda_j\in\mathbb R$ for $1\le j \le r\leq \min\{n_1,\dots,n_k\}$ such that
\begin{equation}\label{eq:orthogonal}
\mathcal A=\sum_{j =1}^r\lambda_j\mathbf u^{(1)}_j\otimes\dots\otimes\mathbf u^{(k)}_j.
\end{equation}
{Here for each $i=1,\dots, k$ and $j=1,\dots, r$, the vector $\mathbf u^{(i)}_j \in \mathbb{R}^{n_i}$ is the $j$-th column vector of the orthonormal matrix $U^{(i)}$.} Without loss of generality, we may assume that $\lambda_j \geq 0$ for all $j =1,\dots,r$.
Throughout this paper, we will always assume that $k\geq 3$.

{Let $r, k$ be positive integers and let $\mathbf n \coloneqq (n_1,\dots,n_k)$ be a $k$-dimensional integer vector}. We denote by $C(\mathbf n, r)\subseteq \mathbb R^{n_1}\otimes\dots\otimes\mathbb R^{n_k}$ the set of all \textit{orthogonally decomposable tensors with rank at most $r$}, i.e.,
\begin{multline}\label{eq:codt-r}
C(\mathbf n,r)\coloneqq \Big\{{\mathcal A}\in\mathbb R^{n_1}\otimes\dots\otimes\mathbb R^{n_k}\colon \mathcal A=(U^{(1)},\dots,U^{(k)})\cdot \operatorname{diag}(\lambda_1,\dots, \lambda_r),\\ U^{(i)}\in V(r,n_i)\ \text{for all }i\in \{1,\dots,k\},\ \lambda_j\in\mathbb R\ \text{for all }  j \in \{1,\dots,r\} \Big\}.
\end{multline}
We also let $D(\mathbf n, r)\subseteq \mathbb R^{n_1}\otimes\dots\otimes\mathbb R^{n_k}$ be the set of all \textit{orthogonally decomposable tensors with rank $r$}, i.e.,
\begin{multline}\label{eq:codt-ex-r}
D(\mathbf n,r):=\Big\{\mathcal A\in\mathbb R^{n_1}\otimes\dots\otimes\mathbb R^{n_k}\colon \mathcal A= (U^{(1)},\dots,U^{(k)})\cdot \operatorname{diag}(\lambda_1,\dots, \lambda_r),\\ U^{(i)}\in V(r,n_i)\ \text{for all }i \in \{1,\dots,k\},\ \lambda_j\neq 0\ \text{for all }  j \in \{1,\dots,r\}\Big\}.
\end{multline}

\begin{lemma}[Unique Decomposition]\label{lem:unique}
For each $\mathcal A\in D(\mathbf n,r)$, {the rank-$r$ decomposition of $\mathcal{A}$ is unique}. In particular, the orthogonal decomposition of $\mathcal A$ is unique.
\end{lemma}

\begin{proof}
It follows from Kruskal's inequality \cite{KB-09,L-12,K-77} immediately. {A direct proof can also be found in \cite{ZG-01}.}
\end{proof}

\subsection{Morse functions}\label{sec:morse}
In the following, we introduce the
notion of Morse functions and {recall some of its basic properties}. {On a smooth manifold $M$, a {smooth} function $f : M\rightarrow \mathbb R$ is called a \textit{Morse function} if each critical point of $f$ on $M$ is nondegenerate,} {i.e., the Hessian matrix of $f$ at each critical point is non-singular.} The following result  is well-known, see for example \cite[Theorem~6.6]{M-63}.
\begin{lemma}[Projection is Generically Morse]\label{lem:generic-morse}
{Let $M$ be a {submanifold} of $\mathbb R^n$. For {a generic} $\mathbf a = (a_1,\dots,a_n)^\tp \in\mathbb R^n$, the Euclidean distance function}
\[
f(\mathbf x)=  \|\mathbf x-\mathbf a\|^2
\]
is a Morse function on $M$.
\end{lemma}

We will also need the following property (cf. \cite[Corollary~2.3]{M-63}) of nondegenerate critical points in the sequel.
\begin{lemma}\label{lem:isolated critical points}
Let $M$ be a manifold and let $f:M \to \mathbb{R}$ be a smooth function. Nondegenerate critical points of $f$ are isolated.
\end{lemma}


{To conclude this subsection, we briefly discuss how critical points behave under a local diffeomorphism. {For this purpose,} we recall that two smooth manifolds $M_1$ and $M_2$ are called \emph{locally diffeomorphic} if there is a smooth map $\varphi : M_1\rightarrow M_2$ such that for each point $\mathbf x\in M_1$ there exists a neighborhood $U\subseteq M_1$ of $\mathbf x$ and a neighborhood $V\subseteq M_2$ of $\varphi(\mathbf x)$ such that the restriction $\varphi|_U : U\rightarrow V$ is a diffeomorphism \cite{dC-92}. In this case, the corresponding $\varphi$ is called a \emph{local diffeomorphism} between $M_1$ and $M_2$. {It is clear from the definition that two locally diffeomorphic manifolds must have the same dimension. Moreover, we have the following result, whose proof can be found in \cite[Proposition~5.2]{HL-18}.}
\begin{proposition}\label{prop:critical}
Let {$M_1$ and $M_2$} be two {locally diffeomorphic} smooth manifolds  and let
$\varphi : M_1\rightarrow M_2$ be the corresponding local diffeomorphism.
Let $f : M_2\rightarrow \mathbb R$ be a smooth function. Then $\mathbf x\in M_1$ is a (nondegenerate) critical point of $f\circ \varphi$ on $M_1$ if and only if $\varphi(\mathbf x)$ is a (nondegenerate) critical point of $f$ on $M_2$.
\end{proposition}

{When $f$ is a smooth function on $\mathbb{R}^n$ and $M$ is a submanifold of $\mathbb{R}^n$, we denote by $\nabla f$ the gradient of $f$ as a function on $\mathbb{R}^n$, while we denote by $\operatorname{grad}(f)$ the Riemannian gradient of $f$ as a function on $M$. In other words, $\operatorname{grad}(f)$ is the projection of $\nabla f$ to the tangent space of $M$.}

\subsection{Kurdyka-\L ojasiewicz property}\label{sec:kl}
In this subsection, we will review some basic facts {about the} Kurdyka-\L ojasiewicz property, which {even holds for nonsmooth functions} in general. {Interested readers are referred} to \cite{ABS-13,LP-16,ABRS-10,BDLM-10}.

Let $p$ be an extended real-valued function and let $\partial p(\mathbf x)$ be the set of sub-differentials of $p$ at $\mathbf x$ (cf.\ \cite{RW-98}). We define $\operatorname{dom}(\partial p) \coloneqq \{\mathbf x\colon\partial p(\mathbf x )\neq \emptyset\}$ and take $\mathbf x^*\in \operatorname{dom}(\partial p)$. If there exist some $\eta\in(0,+\infty]$, a neighborhood $U$ of $\mathbf x^*$, and a continuous concave function $\varphi:[0,\eta)\rightarrow \mathbb{ R}_+$, such that
\begin{enumerate}
\item $\varphi  (0)=0$,
\item $\varphi$ is  $C^1 $ on $(0,\eta)$,
\item for all $s\in (0,\eta)$, $\varphi^{\prime}(s)>0$, and
\item for all $\mathbf x\in U\cap {\{ \mathbf y\colon p(\mathbf x^*)<p(\mathbf y)<p(\mathbf x^*)+\eta \}}$, the Kurdyka-\L ojasiewicz inequality holds
\[
\varphi^{\prime}(p(\mathbf x) - p(\mathbf x^*)) \operatorname{dist}(\mathbf 0,\partial p(\mathbf x))\geq 1,
\]
\end{enumerate}
then {we say that} $p$ has the \emph{Kurdyka-\L ojasiewicz (abbreviated as KL) property} at $\mathbf x^*$.  {Here} $\operatorname{dist}(\mathbf 0,\partial p(\mathbf{x} ))$ denotes the distance from $\mathbf 0$ to the set $\partial p(\mathbf{x} )$. {If $p$ is proper, lower semicontinuous, and has the KL property at each point of $\operatorname{dom}(\partial p)$, then $p$ is said to be a \emph{KL function}.} Examples of KL functions include real subanalytic functions and semi-algebraic functions \cite{BST-14}. In this paper, semi-algebraic functions will be involved, we refer to \cite{BCR-98} and references herein for more details on such functions. {In particular, polynomial functions are semi-algebraic functions and hence KL functions.
Another important fact is that the indicator function of a semi-algebraic set is also a semi-algebraic function \cite{BCR-98,BST-14}. Also, a finite sum of semi-algebraic functions is again semi-algebraic. We assemble these facts to derive the following lemma which will be crucial to the analysis of the global convergence of our algorithm.
\begin{lemma}\label{lem:KL functions}
A finite sum of polynomial functions and indicator functions of semi-algebraic sets is a KL function.
\end{lemma}
}

While KL-property is used for global convergence analysis, the \L ojasiewicz inequality discussed {in the rest of this subsection} is for convergence rate analysis. The classical \L ojasiewicz inequality for analytic functions is {stated as follows} (cf.\ \cite{L-63}):
\begin{itemize}
\item[] {\bf (Classical \L ojasiewicz's gradient inequality)} If $f$ is an analytic function with $f(\mathbf 0)=0$ and $\nabla f(\mathbf 0)=\mathbf 0$, then there exist positive constants $\mu, \kappa,$ and $\epsilon$
such that
\[
\|\nabla f(\mathbf x)\|\ge\mu|f(\mathbf x)|^{\kappa}\;\mbox{ for all }\;\|\mathbf x\|\le\epsilon.
\]
\end{itemize}

As pointed out in \cite{ABRS-10,BDLM-10}, it is often difficult to determine the corresponding exponent $\kappa$ in \L ojasiewicz's gradient inequality, and it is unknown for a general function. However, an estimate of the exponent $\kappa$ in the gradient inequality were derived by  D'Acunto and Kurdyka in \cite[Theorem~4.2]{DK-05} {when $f$ is a polynomial function}. We {record} this fundamental result in the next lemma, which will play a key role in our sublinear convergence rate analysis.

\begin{lemma}[\L ojasiewicz's Gradient Inequality for Polynomials]\label{lemma:loja1}
 Let $f$ be a real polynomial {of degree $d$}. 
Suppose that $f(\mathbf 0)=0$ and $\nabla f(\mathbf 0)=\mathbf 0$. There exist constants $c, \epsilon > 0$ such that for all $\|\mathbf x\|\le\epsilon$, we have
\[
\|\nabla f(\mathbf x)\|\ge c|f(\mathbf x)|^{\kappa}\;\mbox{ with }\;\kappa = 1-\frac{1}{d(3d-3)^{n-1}}.
\]
%
\end{lemma}

Below is a manifold version of the \L ojasiewicz gradient inequality \cite{dC-92}.
\begin{proposition}[\L ojasiewicz's Gradient Inequality]\label{prop:lojasiewicz}
Let $M$ be a smooth manifold and let $g: M \to \mathbb R$ be a smooth function for which $\mathbf z^*$ is a nondegenerate critical point. Then there exists a neighborhood {$ U$} in $M$ of $\mathbf z^*$ such that for all $\mathbf z\in  {U}$
\[
\|\operatorname{grad}(g)(\mathbf z)\|^2\geq \kappa |g(\mathbf z)-g(\mathbf z^*)|
\]
for some $\kappa>0$.
\end{proposition}

\subsection{Low rank orthogonal tensor approximation}\label{sec:lowrank}
The problem considered in this {paper can be described as follows}: given a tensor $\mathcal A\in\mathbb R^{n_1}\otimes\dots\otimes\mathbb R^{n_k}$, find an orthogonally decomposable tensor $\mathcal B\in\mathbb R^{n_1}\otimes\dots\otimes\mathbb R^{n_k}$ of rank at most $r\leq\min\{n_1,\dots,n_k\}$ such that the residual $\|\mathcal A-\mathcal B\|$ is minimized. {More precisely, we will consider the following optimization problem:}
\begin{flalign}\label{eq:sota}
&\text{(LROTA(r))}\ \ \ \ \ \ \begin{array}{rl}\min&\|\mathcal A-(U^{(1)},\dots,U^{(k)})\cdot\Upsilon\|^2\\
\text{s.t.}&\Upsilon=\operatorname{diag}(\upsilon_1,\dots,\upsilon_r),\ \upsilon_i\in\mathbb R,\\
& \big(U^{(i)}\big)^\tp U^{(i)}=I\  \text{for all } {1 \le i \le k}.
\end{array}&
\end{flalign}



\begin{proposition}[Maximization Equivalence]\label{prop:sota-max}
The {optimization} problem~\eqref{eq:sota} is equivalent to
\begin{flalign}\label{eq:sota-max}
&\rm{(mLROTA(r))}\ \ \ \ \ \
\begin{array}{rl}\max&\sum_{j=1}^r\Big(\big(\big(U^{(1)}\big)^\tp ,\dots,\big(U^{(k)}\big)^\tp \big)\cdot\mathcal A\Big)_{j\dots j}^2\\
\rm{s.t.} & \big(U^{(i)}\big)^\tp U^{(i)}=I\ \text{for all } {1 \le i \le k}
\end{array}&
\end{flalign}
in the following sense
\begin{enumerate}
\item if {$ (\mathbb U_*,\Upsilon_* ) \coloneqq \big((U^{(1)}_*,\dots,U^{(k)}_*),  \operatorname{diag}((\upsilon_*)_1,\dots,(\upsilon_*)_r)  \big)$} is an optimizer of \eqref{eq:sota} with the optimal value $\|\mathcal A\|^2-\sum_{i=1}^r(\upsilon_*)_i^2$, then $\mathbb U_*$ is an optimizer of \eqref{eq:sota-max} with the optimal value $\sum_{i=1}^r(\upsilon_*)_i^2$;
\item {conversely}, {if $\mathbb U_*$ is an optimizer of \eqref{eq:sota-max}, then $(\mathbb U_*,\Upsilon_* )$ is an optimizer of \eqref{eq:sota}} where
\[
\Upsilon_*=\operatorname{diag}\Big(\operatorname{Diag}\Big(\big(\big(U_*^{(1)}\big)^\tp ,\dots,\big(U_*^{(k)}\big)^\tp \big)\cdot\mathcal A\Big)\Big).
\]
\end{enumerate}
\end{proposition}

\begin{proof}
{By a direct calculation we may obtain}
\begin{align*}
\|\mathcal A-(U^{(1)},\dots,U^{(k)})\cdot\Upsilon\|^2&=\|\mathcal A\|^2+\sum_{i=1}^r\upsilon_i^2-2\langle\mathcal A,(U^{(1)},\dots,U^{(k)})\cdot\Upsilon\rangle\\
&=\|\mathcal A\|^2+\sum_{i=1}^r\upsilon_i^2-2\Big\langle\big(\big(U^{(1)}\big)^\tp ,\dots,\big(U^{(k)}\big)^\tp \big)\cdot\mathcal A,\Upsilon\Big\rangle\\
&=\|\mathcal A\|^2+\sum_{i=1}^r\upsilon_i^2-2\sum_{i=1}^r\upsilon_i\Big[\big(\big(U^{(1)}\big)^\tp ,\dots,\big(U^{(k)}\big)^\tp \big)\cdot\mathcal A\Big]_{i\dots i}.
\end{align*}
Note that $\upsilon_i$ in the minimization problem \eqref{eq:sota} is unconstrained for all $i\in\{1,\dots,r\}$, and they are mutually independent.
Thus, at an optimizer $(\mathbb U_*,\Upsilon_*) \coloneqq \big((U^{(1)}_*,\dots,U^{(k)}_*),\Upsilon_*\big)$ of \eqref{eq:sota}, we must have
 \begin{equation}\label{eq:lambda}
 (\upsilon_*)_i=\Big[\big(\big(U_*^{(1)}\big)^\tp ,\dots,\big(U_*^{(k)}\big)^\tp \big)\cdot\mathcal A\Big]_{i\dots i}\ \text{for all } 1\le i \le r
 \end{equation}
 and the optimal value is
 \[
 \|\mathcal A\|^2-\sum_{i=1}^r(\upsilon_*)_i^2.
 \]
Therefore, problem~\eqref{eq:sota} is equivalent to \eqref{eq:sota-max}.
\end{proof}

\section{KKT Points via Projection onto Manifolds}\label{sec:manifold}
{On the one hand, a numerical algorithm solving the optimization problem \eqref{eq:sota} (or equivalently its maximization reformulation \eqref{eq:sota-max}) is usually designed} in the {parameter} space
\[
V_{\mathbf n,r}:=V(r,n_1)\times \dots\times V(r,n_k)\times\mathbb R^r.
\]
{See for example, \cite{CS-09,GC-19,Y-19,WTY-15}.}
On the other hand, {from a more geometric perspective, we can also regard problem \eqref{eq:sota} as} the projection of a given tensor $\mathcal A$ onto $C(\mathbf n,r)$. A key ingredient {in our study of problem \eqref{eq:sota}} is the {relation} between these two {viewpoints}. {Once such a connection is {understood}, we will be able to derive an algorithm in $V_{\mathbf{n},r}$ but analyse it in $C(\mathbf n,r)$.} {To be more precise, we will study both the problem of the projection
\begin{equation}\label{eq:projection}
\begin{array}{rl}
\min& \|\mathcal A-\mathcal B\|^2\\
\text{s.t.}&\mathcal B\in D(\mathbf n,r),
\end{array}
\end{equation}
and its parametrization
\begin{flalign}\label{eq:sota-proj}
&\text{(LROTA-P)}\ \ \ \ \ \ \begin{array}{rl}\min&g(\mathbb U,\mathbf x)=\frac{1}{2}\|\mathcal A-(U^{(1)},\dots,U^{(k)})\cdot\operatorname{diag}(\mathbf x)\|^2\\
\text{s.t.}
& \big(U^{(i)}\big)^\tp U^{(i)}=I\  \text{for all }1\le i \le k,\\
& \mathbf x\in\mathbb R^r_*,
\end{array}&
\end{flalign}
{where $\mathbb R_*:=\mathbb R\setminus\{0\}$.}
}

{We will first study properties of $C(\mathbf n,r)$ and $D(\mathbf{n},r)$ and then discuss critical points of problem~\eqref{eq:projection} in Section~\ref{subsec:geometry of codt}. KKT points of \eqref{eq:sota-max} and hence \eqref{eq:sota-proj} will be discussed in Section~\ref{sec:kkt}. The connection between them will be studied in Section~\ref{sec:critical-kkt}, in which a \L ojasiwicz inequality for KKT points of \eqref{eq:sota-max} will be given. We refer to \cite{dC-92,S-77,M-63,H-77} for basic facts of differential geometry, algebraic geometry and algebraic topology that will be used in the sequel.}

\subsection{Geometry of  orthogonally decomposable tensors}\label{subsec:geometry of codt}

Let
\begin{equation}\label{eq:paraset}
U_{\mathbf{n},r} \coloneqq V(r,n_1)\times \cdots \times V(r,n_k) \times \mathbb{R}_*^r.
\end{equation}
By the next proposition, $U_{\mathbf{n},r}$ parametrizes the manifold $D(\mathbf n,r)$.
\begin{proposition}\label{prop:smooth}
For each positive integer $r\leq\min\{n_1,\dots,n_k\}$, the map
\begin{align*}
\varphi_{\mathbf{n},r}: V(r,n_1)\times \cdots \times V(r,n_k) \times \mathbb{R}^r &\to C(\mathbf{n},r),\\
 (U^{(1)},\dots, U^{(k)}, (\lambda_1,\dots,\lambda_r)) &\mapsto   (U^{(1)},\dots,U^{(k)})\cdot \operatorname{diag}(\lambda_1,\dots, \lambda_r)
\end{align*}
is a surjective map and we have the following:
\begin{itemize}
\item The permutation group $\mathfrak{S}_r$ acts on $V_{\mathbf{n},r}$ such that $\varphi_{\mathbf{n},r}$ is $\mathfrak{S}_r$-invariant.
\item The inverse image $U_{\mathbf{n},r}=\varphi_{\mathbf{n},r}^{-1} (D(\mathbf n, r)) \subseteq V_{\mathbf{n},r}$ consists of points
\[
\big(U^{(1)},\dots, U^{(k)},(\lambda_1,\dots, \lambda_r)\big)
\]
such that $\lambda_j \ne 0$ for all $1\le j \le r$. In particular, $U_{\mathbf{n},r}$ is an open submanifold of $V_{\mathbf{n},r}$.
\item {$U_{\mathbf{n},r}$} is a principal $\mathfrak{S}_r$-bundle on $D(\mathbf{n},r)$, i.e., we have
\[
U_{\mathbf{n},r}/\mathfrak{S}_r \simeq D(\mathbf{n},r).
\]
\item $D(\mathbf n, r)$ is a smooth manifold of dimension
\[
d_{\mathbf{n},r} \coloneqq r\left[\sum_{i=1}^kn_i-\frac{k(r+1)}{2}+1\right].
\]
\item $C(\mathbf{n},r) = \bigsqcup_{t = 0}^r D(\mathbf{n},t)$ is an irreducible algebraic variety whose singularity locus is $\bigsqcup_{t = 0}^{r-1} D(\mathbf{n},t)$. In particular, $C(\mathbf{n},r) $ has dimension $d_{\mathbf{n},r}$.
\end{itemize}
\end{proposition}

\begin{proof}
We recall that $V(r,n)$ consists of all $n\times r$ matrices whose columns are orthonormal. Hence $V(r,n)$ admits an $\mathfrak{S}_r$ action by permuting columns. In other words, an element $\sigma$ in $\mathfrak{S}_r$ can be written as an $r\times r$ permutation matrix $P_\sigma$, the action of $\mathfrak{S}_r$ on $V(r,n)$ is simply given by
\[
\mathfrak{S}_r \times V(r,n) \to V(r,n),\quad (\sigma,U) \mapsto U P_\sigma,
\]
and this induces an action
\begin{align*}
\mathfrak{S}_r \times \left(
V(r,n_1)\times \cdots \times V(r,n_k) \times \mathbb{R}^r
\right) &\to \left(
V(r,n_1)\times \cdots \times V(r,n_k) \times \mathbb{R}^r
\right) \\
(\sigma,(U^{(1)},\dots, U^{(k)},(\lambda_1,\dots,\lambda_r)) ) &\mapsto (U^{(1)}P_{\sigma},\dots,U^{(k)}P_{\sigma}, (\lambda_{\sigma(1)},\dots, \lambda_{\sigma(r)})).
\end{align*}
Now it is straightforward to verify that $\varphi_{\mathbf{n},r}$ is $\mathfrak{S}_r$-invariant.

Since $D(\mathbf{n},r)$ consists of all  orthogonally decomposable tensors with rank exactly $r$, its inverse image $U_{\mathbf{n},r} $ cannot contain a point of the form
\[
(U^{(1)},\dots,U^{(k)},(\lambda_1,\dots, \lambda_r))
\]
where $\lambda_j = 0$ for some $1\le j \le r$ by Lemma~\ref{lem:unique}. Moreover, we claim that any tensor of the form
\[
\mathcal{T} = \varphi_{\mathbf{n},r} \left( U^{(1)},\dots,U^{(k)},(\lambda_1,\dots, \lambda_r) \right) = \left( U^{(1)},\dots,U^{(k)} \right) \cdot \operatorname{diag}(\lambda_1,\dots, \lambda_r)
\]
where $\lambda_j \ne 0,1\le j \le r$, must lie in $D(\mathbf{n},r)$. Indeed, by definition, we see that $\mathcal{T}$ has rank at most $r$. Moreover, by the orthogonality of column vectors of each $U^{(j)},j=1,\dots, r$, the mode-$1$ matrix flattening $T^{(1)}\in \mathbb{R}^{n_1 \times (n_2 \cdots n_k)}$ of $\mathcal{T}$ has rank $r$, which implies that $\mathcal{T}$ has rank at least $r$ and hence $\mathcal{T}$ has rank $r
$ \cite{L-12,L-13}. This implies that $U_{\mathbf{n},r} $ is an open subset and hence an open submanifold of $V_{\mathbf{n},r}$.

We notice that $U_{\mathbf{n},r} $ admits an $\mathfrak{S}_r$-action by the restriction of that on $V_{\mathbf{n},r}$ and the fiber $\varphi_{\mathbf{n},r}^{-1}(\mathcal{T})\simeq \mathfrak{S}_r$ if $\mathcal{T}\in D(\mathbf{n},r)$. This implies that $U_{\mathbf{n},r} /\mathfrak{S}_r \simeq D(\mathbf{n},r)$.

Since $\mathfrak{S}_r$ is a finite group acting on $U_{\mathbf{n},r} $ freely, we conclude that $D(\mathbf{n},r) \simeq U_{\mathbf{n},r} /\mathfrak{S}_r$ is a smooth manifold whose dimension is
\[
\dim D(\mathbf{n},r) =\dim U_{\mathbf{n},r}  = \dim V_{\mathbf{n},r} = \sum_{j=1}^k \dim V(r,n_j) + \dim \mathbb{R}^r.
\]
Observing that $\dim V(r,n) = r(n-r) + \binom{r}{2}$, we obtain the desired formula for $\dim D(\mathbf{n},r) $.

The fact that $C(\mathbf{n},r)$ is an algebraic variety follows directly from the definition. Since $C(\mathbf{n},r)$ is the image of the irreducible algebraic variety $V(\mathbf{n},r)$ under the map $\varphi_{\mathbf{n},r}$, we may conclude that $C(\mathbf{n},r)$ is irreducible. It is straightforward to verify that the rank of the differential $d\varphi_{\mathbf{n},r}$ drops at points in $\bigsqcup_{t=0}^{r-1} D(\mathbf{n},t)$ and this implies that $\bigsqcup_{t=0}^{r-1} D(\mathbf{n},t)$ is the singular locus of $C(\mathbf{n},r)$.
\end{proof}

 We show in the next lemma that $U_{\mathbf n,r}$ is locally diffeomorphic to $D(\mathbf n,r)$.

\begin{lemma}[Local Diffeomorphism]\label{lem:localdiff}
For any positive integers $n_1,\dots,n_k$ and $r\leq \min\{n_1,\dots,n_k\}$, the set $U_{\mathbf n,r}$ is a smooth manifold and is locally diffeomorphic to the manifold $D(\mathbf n,r)$.
\end{lemma}

\begin{proof}
We recall from Proposition~\ref{prop:smooth} that $U_{\mathbf{n},r}$ is a principle $\mathfrak{S}_r$-bundle on $D(\mathbf{n},r)$. In particular, since $\mathfrak{S}_r$ is a finite group, for any $\mathcal{T}\in D(\mathbf{n},r)$, the fiber $\varphi_{\mathbf{n},r}^{-1}(\mathcal{T})$ of the map
\[
\varphi_{\mathbf{n},r}: U_{\mathbf{n},r} \to D(\mathbf{n},r)
\]
consists of $r!$ points. Therefore, for a small enough neighbourhood $W\subseteq D(\mathbf{n},r)$ of $\mathcal{T}$, the inverse image $\varphi_{\mathbf{n},r}^{-1}(W)$ is the disjoint union of $r!$ open subsets $W_1,\dots, W_{r!}\subseteq U_{\mathbf{n},r}$ and for each $j=1,\dots, r!$, we have
\[
(\varphi_{\mathbf{n},r})|_{W_j}: W_j \to W
\]
is a diffeomorphism.
\end{proof}

By Lemma~\ref{lem:localdiff} and Proposition~\ref{prop:critical}, problems on $D(\mathbf n,r)$ can be studied via problems on $U_{\mathbf n,r}$. To that end, the tangent space of $U_{\mathbf n,r}$ will be given {at first}. The following result can be checked directly, see \cite{AMS-08,EAT-98}.
\begin{proposition}[Tangent Space of $U_{\mathbf n, r}$]\label{prop:tangent}
At any point $(\mathbb U,\mathbf x)\in U_{\mathbf n,r}$, the tangent space of $U_{\mathbf n, r}$ at $(\mathbb U,\mathbf x)$ is
{\begin{equation}\label{eq:tangent-para}
\operatorname{T}_{(\mathbb U,\mathbf x)}U_{\mathbf n, r}=\operatorname{T}_{U^{(1)}} V(r,n_1)\times\dots\times\operatorname{T}_{U^{(k)}}V(r,n_k) \times\mathbb R^r,
\end{equation}
where $\operatorname{T}_{U^{(i)}} V(r,n_i)$ is the tangent space of the Stiefel manifold $V(r,n_i)$ at $U^{(i)}$, which is
\begin{equation}\label{eq:stiefel-tangent}
\operatorname{T}_{U^{(i)}} V(r,n_i)=\{Z\in\mathbb R^{n_i\times r}\colon (U^{(i)})^\tp Z+Z^\tp U^{(i)}=0\},
\end{equation}
}
for all $i=1,\dots,k$.
\end{proposition}

We can embed $U_{\mathbf n,r}$ into $\mathbb R^{n_1\times r}\times\dots\times\mathbb R^{n_k\times r}\times\mathbb R^r$ in an obvious way and hence $U_{\mathbf n,r}$ becomes an embedded submanifold of the latter. For a differentiable function $f : U_{\mathbf n,r}\subseteq \mathbb R^{n_1\times r}\times\dots\times\mathbb R^{n_k\times r}\times\mathbb R^r\to \mathbb R$, a critical point $(\mathbb U,\mathbf x)$ is a point at which the Riemannian gradient $\operatorname{grad} (f)(\mathbb U,\mathbf x)$ {of $f$ at {$(\mathbb U,\mathbf{x})$} is zero, which is equivalent to {the fact that} the projection of the Euclidean gradient $\nabla f(\mathbb U,\mathbf x)$ onto the tangent space of $U_{\mathbf n,r}$ at $(\mathbb U,\mathbf x)$ is zero.
More explicitly, we have the following characterization.

\begin{lemma}\label{lem:grand}
Let $A\in V(r,n)$ and let $f : V(r,n)\subseteq \mathbb R^{n\times r}\rightarrow \mathbb R$ be {a smooth function.} Then $\operatorname{grad}(f)(A)=0$ if and only if
\begin{equation}\label{eq:symmetry}
\nabla f(A) = A (\nabla f(A))^\tp A,
\end{equation}
{which is {also} equivalent to $\nabla f(A)=AP$ for {some} symmetric matrix $P\in\operatorname{ S}^{r\times r}$. In particular, $A^\tp \nabla f(A)$ is a symmetric matrix.}
\end{lemma}
\begin{proof}
{The proof of the first equivalence can be found in \cite[Proposition~1]{LWS-16}. For the second, we notice that from \eqref{eq:symmetry}
\[
A^\tp \nabla f(A)  = A^\tp A (\nabla f(A))^\tp A = \nabla f(A)^\tp A
\]
and this proves that $A^\tp \nabla f(A)$ is symmetric. Now if we set $P \coloneqq A^\tp \nabla f(A) $ then \eqref{eq:symmetry} can be written as
\[
\nabla f(A) = AP^\tp = AP.
\]
Conversely, if $\nabla f(A)  = AP$ for some symmetric matrix $P$, then \eqref{eq:symmetry} obviously holds by the symmetry of $P$ and the fact that $A^\tp A = I$.}
\end{proof}

Let
\begin{equation}\label{eq:objective-p}
g(\mathbb U,\mathbf x):=\frac{1}{2}\|\mathcal A-(U^{(1)},\dots,U^{(k)})\cdot\operatorname{diag}(\mathbf x)\|^2
\end{equation}
be the objective function of \eqref{eq:sota-proj}. {Since the feasible set {$D(\mathbf n,r)$} (resp. $U_{\mathbf n,r}$) of \eqref{eq:projection} (resp. \eqref{eq:sota-proj})} is a smooth manifold,
we may apply Proposition~\ref{prop:critical}, Lemmas~\ref{lem:generic-morse} and \ref{lem:localdiff} to obtain the following
\begin{proposition}\label{prop:nondegnerate}
For a generic tensor $\mathcal A$, each critical point of the function $g$ on $U_{\mathbf n,r}$ is nondegenerate.
\end{proposition}

\subsection{KKT points of LROTA}\label{sec:kkt}
In this subsection, we will derive the KKT system {of} the optimization problem \eqref{eq:sota-max} and study its properties.

\subsubsection{Existence}\label{sec:kkt-existence}
{
Let $\mathbb U \coloneqq (U^{(1)},\dots,U^{(k)})$ be the collection of the variable matrices in \eqref{eq:sota-max} and for each $1 \le i \le k$ and $1 \le j \le r$, let $\mathbf u^{(i)}_j$ be the $j$-th column of the matrix $U^{(i)}$ and let
\[
\mathbf x_j \coloneqq (\mathbf u^{(1)}_j,\dots,\mathbf u^{(k)}_j)\ \text{and }\mathbf v^{(i)}_j \coloneqq \mathcal A\tau_i(\mathbf x_j).
\]
For {each} $1 \le i \le k$, {we define a matrix}
\begin{equation}\label{eq:critial}
V^{(i)}:=\begin{bmatrix}\mathbf v^{(i)}_1&\dots&\mathbf v^{(i)}_r\end{bmatrix},
\end{equation}
{and a diagonal matrix}
\begin{equation}\label{eq:lambda-tensor}
\Lambda \coloneqq \operatorname{diag}(\mathcal A\tau(\mathbf x_1),\dots,\mathcal A\tau(\mathbf x_r)).
\end{equation}
{{For each} $1 \le j \le r$, we also set}
\begin{equation}\label{eq:lambda}
\lambda_j(\mathbb U) \coloneqq \Big(\big(\big(U^{(1)}\big)^\tp ,\dots,\big(U^{(k)}\big)^\tp \big)\cdot\mathcal A\Big)_{j\dots j}=\mathcal A\tau(\mathbf x_j)=\langle\mathcal A,\mathbf u^{(1)}_j\otimes\dots\otimes\mathbf u^{(k)}_j\rangle
\end{equation}
Now the objective function of \eqref{eq:sota-max} can be  re-written as
\begin{equation}\label{eq:objective}
f(\mathbb U) \coloneqq
\sum_{j=1}^r\Big(\big(\big(U^{(1)}\big)^\tp ,\dots,\big(U^{(k)}\big)^\tp \big)\cdot\mathcal A\Big)_{j\dots j}^2 = \sum_{j=1}^r\lambda_j(\mathbb U)^2.
\end{equation}
}
{\begin{definition}[KKT point]\label{def:kkt}
Let $\mathbb U =(U^{(1)},\dots,U^{(k)})$ be a feasible point  of \eqref{eq:sota-max}. If there exists $\mathbb P=(P_1,\dots,P_k)$ where $P_i\in \operatorname{S}^{r\times r}$ for each $1\le i \le k$ such that the system
\begin{equation}\label{eq:kkt}
V^{(i)}\Lambda-U^{(i)}P_i=0,\quad 1 \le i \le k.
\end{equation}
is satisfied, then $\mathbb{U}$ is called a Karush-Kuhn-Tucker point (KKT point) and $\mathbb{P}$ is called a Lagrange multiplier associated to $\mathbb{U}$. The set of all multipliers associated to $\mathbb U$ is denoted by $\operatorname{M}(\mathbb U)$.
\end{definition}
}
{It follows immediately from the system} \eqref{eq:kkt} that for all $1\le i \le k$,
\begin{equation}\label{eq:kkt-symmetry}
(U^{(i)})^\tp V^{(i)}\Lambda=
\big(V^{(i)}\Lambda\big)^\tp U^{(i)}.
\end{equation}

For an equality constrained optimization problem, we say that a feasible point satisfies \textit{linear independence constraint qualification} (LICQ) if at this point all the gradients of the constraints are {linearly} independent.
\begin{proposition}[LICQ]\label{prop:licq}
At any feasible point of the problem \eqref{eq:sota-max}, LICQ is satisfied. Thus, at any local maximizer of \eqref{eq:sota-max}, the system of KKT condition holds and $\operatorname{M}(\mathbb U)$ is a singleton.
\end{proposition}

\begin{proof}
Suppose on the contrary that LICQ is not satisfied at a feasible point $\mathbb U:=(U^{(1)},\dots,U^{(k)})$ of \eqref{eq:sota-max}. Let $P_i\in\mathbb R^{r\times r}$ for $i=1,\dots,k$ be the corresponding multipliers for the equality constraints in \eqref{eq:sota-max} such that they are not all zero. {To be more precise, $P_i$'s are defined by}
\[
\nabla_{\mathbb U} \bigg(\sum_{i=1}^k\langle \big(U^{(i)}\big)^\tp U^{(i)}-I,P_i\rangle \bigg)=0.
\]
Aligning along the natural block partition as $\mathbb U$,
we must have
\begin{equation}\label{eq:kkt-2}
\langle U^{(i)}P_i,M^{(i)}\rangle=0,\quad M^{(i)}\in\mathbb R^{n_i\times r}, 1\le i \le k.
\end{equation}
Now from \eqref{eq:kkt-2} we obtain $U^{(i)}P_i=0$ and hence $P_i=0$ by the orthonormality of $U^{(i)}$ for all $1 \le i \le k$, {and this contradicts the assumption that not all $P_i$'s are zero. Therefore, LICQ is satisfied, which implies the uniqueness of the multiplier.} The rest conclusion follows from the classical theory of KKT condition \cite{B-99}.
\end{proof}

\begin{lemma}\label{lem:kkt-equiv}
A feasible point $(\mathbb U,\Upsilon)$ is a KKT point of problem \eqref{eq:sota} with multiplier $\mathbb P\coloneqq (P^{(1)},\dots,P^{(k)})$ if and only if $\mathbb U$ is a KKT point of problem \eqref{eq:sota-max} with multiplier $\mathbb P$ and $\Upsilon = \operatorname{diag}\big(\operatorname{Diag}\big(((U^{(1)})^\tp ,\dots,(U^{(k)})^\tp )\cdot\mathcal A\big)\big)$.
\end{lemma}

\begin{proof}
{According to Proposition~\ref{prop:sota-max}, problem \eqref{eq:sota}  is equivalent to \eqref{eq:sota-max}, from which we may obtain the desired correspondence between KKT points.}
\end{proof}

\subsubsection{Primitive KKT points and essential KKT points}\label{sec:degenerate-kkt}
{
It is possible that for some $1 \le j \le r$, $v_j$ approaches to zero along iterations of an algorithm solving the problem \eqref{eq:sota}.} In this case, the {resulting}  orthogonally decomposable tensor is of rank strictly smaller than $r$. We will discuss this degenerate case in this section.
\begin{proposition}[KKT Reduction]\label{prop:degenerate}
Let $\mathbb U=(U^{(1)},\dots,U^{(k)})\in V(r,n_1)\times\dots\times V(r,n_k)$ be a KKT point of problem mLROTA(r) defined in \eqref{eq:sota-max} {and let $1 \le j \le r$ be a fixed integer. Set
\[
\hat{\mathbb U} \coloneqq (\hat U^{(1)},\dots,\hat U^{(k)})\in V(r-1,n_1)\times\dots\times V(r-1,n_k),
\]
where for each $1\le i \le k$, $\hat U^{(i)}$ is the matrix obtained by deleting the {$j$-th} column of $U^{(i)}$. If $\mathcal A\tau(\mathbf x_j)=0$, then $\hat{\mathbb U}$ is a KKT point of the problem mLROTA(r-1):}
\begin{equation}\label{eq:sota-max-matrix-r}
\begin{array}{rl}\max&\|\operatorname{Diag}\big(\big(\big(U^{(1)}\big)^\tp ,\dots,\big(U^{(k)}\big)^\tp \big)\cdot\mathcal A\big)\|^2\\
\rm{s.t.} & \big(U^{(i)}\big)^\tp U^{(i)}=I,\ U^{(i)}\in \mathbb R^{n_i\times (r-1)} , 1\le i \le k.
\end{array}
\end{equation}
\end{proposition}

\begin{proof}
By \eqref{eq:kkt}, the KKT system of problem \eqref{eq:sota-max} is
\[
V^{(i)}\Lambda = U^{(i)}P_i\ \text{for all }i=1,\dots,k,
\]
where {$(P_1,\dots,P_k) \in \operatorname{S}^{r\times r} \times \cdots \times \operatorname{S}^{r\times r}$} is the associated Lagrange multiplier. Without loss of generality, we may assume that $j=r$, which implies that the last diagonal element of $\Lambda$ is zero. Thus,
\[
(U^{(i)})^\tp \begin{bmatrix}\mathbf v^{(i)}_1&\dots&\mathbf v^{(i)}_{r-1}&\mathbf v^{(i)}_r\end{bmatrix}\begin{bmatrix}\hat \Lambda&\mathbf 0\\ \mathbf 0& 0\end{bmatrix}=P_i,\quad 1\le i \le k,
\]
where $\hat\Lambda$ is the leading $(r-1)\times (r-1)$ principal submatrix of $\Lambda$.
This implies that the last column of $P_i$ is zero. By the symmetry of $P_i$, we conclude that $P_i$ is in a block diagonal {form} with
\[
P_i=\begin{bmatrix}\hat P_i&\mathbf 0\\ \mathbf 0& 0\end{bmatrix}, \quad 1\le i \le k.
\]
Therefore we have
\[
\begin{bmatrix}\mathbf v^{(i)}_1&\dots&\mathbf v^{(i)}_{r-1}&\mathbf 0\end{bmatrix}\begin{bmatrix}\hat \Lambda&\mathbf 0\\ \mathbf 0&0\end{bmatrix}=\begin{bmatrix}\hat U^{(i)}&\mathbf u^{(i)}_r\end{bmatrix}\begin{bmatrix}\hat P_i&\mathbf 0\\ \mathbf 0& 0\end{bmatrix},\quad 1\le i \le k,
\]
which implies
\[
\begin{bmatrix}\mathbf v^{(i)}_1&\dots&\mathbf v^{(i)}_{r-1}\end{bmatrix}\hat\Lambda=\hat U^{(i)}\hat P_i,\quad 1\le i \le k.
\]
Consequently, we {may} conclude that $\hat{\mathbb U}$ is a KKT point of \eqref{eq:sota-max-matrix-r}.
\end{proof}

A KKT point $\mathbb U=(U^{(1)},\dots,U^{(k)})\in V(r,n_1)\times\dots\times V(r,n_k)$ of problem mLROTA(r) (cf.\ \eqref{eq:sota-max}) with $\mathcal A\tau(\mathbf x_j)\neq 0$ for all $1\le j \le r$ is called a \textit{primitive KKT point} of mLROTA(r). {Iteratively applying Proposition~\ref{prop:degenerate}, we obtain the following
\begin{corollary}\label{cor:essential-primitive}
Let $S$ be a proper subset of $\{1,\dots,r\}$ with cardinality $s:=|S|<r$ and let $\mathbb U=(U^{(1)},\dots,U^{(k)})\in V(r,n_1)\times\dots\times V(r,n_k)$ be a KKT point of mLROTA(r). Set
\[
\hat{\mathbb U} \coloneqq (\hat U^{(1)},\dots,\hat U^{(k)})\in V(r-s,n_1)\times\dots\times V(r-s,n_k),
\]
where for each $1 \le i \le k$, $\hat U^{(i)}$ is obtained by deleting those columns indexed by $S$. If $\mathcal A\tau(\mathbf x_j)=0$ for all $j\in S$, then $\hat{\mathbb U}$ is a primitive KKT point of mLROTA(r-s).
\end{corollary}
}

It would happen that several KKT points of mLROTA(r) reduce in this way to the same primitive KKT point of mLROTA(r-s). {We call the set of such KKT points an \emph{essential KKT point}.} Therefore, there is a one to one correspondence between essential KKT points of mLROTA(r) and all primitive KKT points of mLROTA(s) for $1 \le s \le r$.

\subsection{Critical points are KKT points}\label{sec:critical-kkt}
In this section, we will establish the relation between KKT points of problem \eqref{eq:sota-max} and critical points of {$g$ on the manifold $U_{\mathbf n,r}$, which is the objective function defined in \eqref{eq:objective-p}.} To do this, we recall {from \eqref{eq:tangent-form}} that the gradient of $g$ at a point $(\mathbb U,\mathbf x)\in U_{\mathbf n,r}$ is given by
\begin{align}
\operatorname{grad}_{U^{(i)}}g(\mathbb U,\mathbf x)&=(I-\frac{1}{2}U^{(i)}(U^{(i)})^\tp )\big(V^{(i)}\Gamma-U^{(i)}(V^{(i)}\Gamma)^\tp  U^{(i)}\big),\label{eq:gradient-ui}\\
\operatorname{grad}_{\mathbf x}g(\mathbb U,\mathbf x)&=\mathbf x-\operatorname{Diag}\big(((U^{(1)})^\tp ,\dots,(U^{(k)})^\tp )\cdot\mathcal A\big),\label{eq:gradient-x}
\end{align}
where {$i=1,\dots, k$ and} $\Gamma=\operatorname{diag}(\mathbf x)$ is the diagonal matrix formed by the vector $\mathbf x$.

\begin{proposition}\label{prop:critical-equivalence}
A point $(\mathbb U,\mathbf x)\in U_{\mathbf n,r}$ is a critical point of $g$ defined in \eqref{eq:objective-p} if and only if $(\mathbb U,\operatorname{diag}(\mathbf x))$ is a KKT point of problem \eqref{eq:sota}.
\end{proposition}

\begin{proof}
We recall that a critical point $(\mathbb U,\mathbf x)\in U_{\mathbf n,r}$ of $g$ is defined by $\operatorname{grad}(g)(\mathbb U,\mathbf x)=0$. It follows from Proposition~\ref{prop:tangent} that these critical points are defined by
\[
\nabla_{U^{(i)}} g(\mathbb U,\mathbf x)=U^{(i)}P_i,\quad 1\le i \le r
\]
where $P_i$ is some $r\times r$ symmetric matrix and
\[
\nabla_{\mathbf x} g(\mathbb U,\mathbf x)=0.
\]
By \eqref{eq:gradient-x}, we have
\[
\mathbf x = \operatorname{Diag}\big(((U^{(1)})^\tp ,\dots,(U^{(k)})^\tp )\cdot\mathcal A\big),
\]
and according to \eqref{eq:gradient-ui}, we obtain
\[
\nabla_{U^{(i)}} g(\mathbb U,\mathbf x)=V^{(i)}\Lambda.
\]
Therefore, by \eqref{eq:kkt} a critical point of $g$ {on $U_{\mathbf n,r}$ must come from} a KKT point of problem \eqref{eq:sota}. The converse is obvious {and this completes the proof}.
\end{proof}

\begin{definition}[Nondegenerate KKT Point]\label{def:nondegenerate}
A KKT point $(\mathbb U,\Upsilon)$ of problem \eqref{eq:sota} is nondegenerate  if $(\mathbb U,\mathbf x)\in U_{\mathbf n,r}$ is a nondegenerate critical point of $g$ with $\operatorname{diag}(\mathbf x)=\Upsilon$.
\end{definition}

\begin{theorem}[Finite Essential Critical Points]\label{thm:kkt}
For a generic tensor, there are only finitely many essential KKT points for problem~\eqref{eq:sota-max}, {and for any positive integers $r > s > 0$,
a primitive KKT point of the problem mLROTA(s) corresponding to an essential KKT point  of the problem mLROTA(r) is nondegenerate.}
\end{theorem}

\begin{proof}
{To prove the finiteness of essential KKT points, it is sufficient to show that there are only finitely many primitive KKT points on $U_{\mathbf n,r}$, and the finiteness follows from the layer structure of the set $V_{\mathbf n,r}$ (or equivalently  $C(\mathbf n,r)$, cf.\ Proposition~\ref{prop:smooth}) and Proposition~\ref{prop:degenerate}.} We first recall that KKT points on $U_{\mathbf n,r}$ are defined by \eqref{eq:kkt}, which is a system of polynomial equations, which implies that the set $K_{\mathbf{n},r}$ of KKT points of problem \eqref{eq:sota-max} on $U_{\mathbf n,r}$ is a closed subvariety of the quasi-variety $U_{\mathbf{n},r}$. We also note that there are finitely many irreducible components of $K_{\mathbf n,r}$ \cite{H-77} and hence it suffices to prove that each irreducible component of $K_{\mathbf{n},r}$ is a singleton. Now let $Z\subseteq K_{\mathbf{n},r}$ be an irreducible component of $K_{\mathbf{n},r}$. If $Z$ contains infinitely many points, then $\dim Z \ge 1$ \cite{H-77}. {However, each point in $Z$ determines a critical point of the function $g$ defined in \eqref{eq:objective-p} on the manifold $U_{\mathbf{n},r}$ (cf.\ Proposition~\ref{prop:critical-equivalence}). This implies that the set of critical points of $g$ on $U_{\mathbf{n},r}$ has a positive dimension, which contradicts Lemma~\ref{lem:isolated critical points} and Proposition~\ref{prop:nondegnerate}}.

Next, by Corollary~\ref{cor:essential-primitive}, given a non-primitive KKT point $\mathbb U$ of the problem mLROTA(r), we can get a primitive KKT point $\hat{\mathbb U}$ of problem mLROTA(s) with $s < r$
{and hence we have} $(\hat{\mathbb U},\mathbf x)\in U_{\mathbf n,s}$ where $\mathbf x$ is determined by $\lambda_j(\hat{\mathbb U})$'s. Since for a generic tensor the function $g$ has only nondegenerate critical points on $U_{\mathbf n,s}$ by Proposition~\ref{prop:nondegnerate}, the second assertion follows from Proposition~\ref{prop:critical-equivalence} and Corollary~\ref{cor:essential-primitive}.
\end{proof}

{For simplicity, we abbreviate $\nabla_{U^{(i)}} f(\mathbb U)$ as $\nabla_i f(\mathbb U)$ for each $1 \le i \le k$.
We define
\[
\|\mathbb U\|_F^2 \coloneqq \sum_{i=1}^k\|U^{(i)}\|_F^2.
\]
The following result is crucial to the linear convergence analysis in the sequel.}
\begin{lemma}[\L ojasiwicz's Inequality]\label{lem:gradient}
If $(\mathbb U^*,\Upsilon^*)$ is a nondegenerate KKT point of problem \eqref{eq:sota}, then there exist $\kappa>0$ and $\epsilon>0$ such that
\begin{equation}
\sum_{i=1}^k\|\nabla_i f(\mathbb U)-U^{(i)}\nabla_i (f(\mathbb U))^\tp U^{(i)}\|_F^2\geq \kappa |f(\mathbb U)-f(\mathbb U^*)|
\end{equation}
for any $\|\mathbb U-\mathbb U^*\|_F\leq\epsilon$. {Here $f$ is the objective function of problem \eqref{eq:sota-max} defined by \eqref{eq:objective}.}
\end{lemma}

\begin{proof}
Let $\delta>0$ be the radius of the neighborhood given by Proposition~\ref{prop:lojasiewicz}.
Since $(\mathbb U^*,\Upsilon^*)$ is a KKT point of \eqref{eq:sota}, we have
\[
\operatorname{Diag}(\Upsilon^*)=\operatorname{Diag}\big(((U^{(*,1)})^\tp ,\dots,(U^{(*,k)})^\tp )\cdot\mathcal A\big).
\]

For a given $\mathbb U$, let $\Upsilon$ be defined as
\[
\operatorname{Diag}(\Upsilon)=\operatorname{Diag}\big(((U^{(1)})^\tp ,\dots,(U^{(k)})^\tp )\cdot\mathcal A\big).
\]
Then, there exists $\epsilon>0$ such that
\begin{equation}\label{eq:upsilon}
\|(\mathbb U,\Upsilon)-(\mathbb U^*,\Upsilon^*)\|_F\leq \delta
\end{equation}
whenever $\|\mathbb U-\mathbb U^*\|_F\leq \epsilon$.

By Proposition~\ref{prop:critical-equivalence}, Proposition~\ref{prop:lojasiewicz} is applicable to $(\mathbb U^*,\mathbf x^*)\in U_{\mathbf n,r}$ for the function $g$.
Thus, there exists $\kappa_0>0$ such that
\[
\|\operatorname{grad}(g)(\mathbb U,\mathbf x)\|^2\geq \kappa_0 |g(\mathbb U,\mathbf x)-g(\mathbb U^*,\mathbf x^*)|
\]
for all $\|\mathbb U-\mathbb U^*\|_F\leq \epsilon$, where $\mathbf x=\operatorname{Diag}(\Upsilon)$ is formed by the diagonal elements of $\Upsilon$.

We first have
\[
|g(\mathbb U,\mathbf x)-g(\mathbb U^*,\mathbf x^*)|= {\frac{1}{2}} |f(\mathbb U)-f(\mathbb U^*)|,
\]
since $f(\mathbb U)=\|\mathbf x\|^2$ and
\begin{align*}
g(\mathbb U,\mathbf x)&=\frac{1}{2}\|\mathcal A-(U^{(1)},\dots,U^{(k)})\cdot\operatorname{diag}(\mathbf x)\|^2\\
&=\frac{1}{2}\|\mathcal A\|^2-\langle \mathcal A,(U^{(1)},\dots,U^{(k)})\cdot\operatorname{diag}(\mathbf x)\rangle+\frac{1}{2}\|\mathbf x\|^2\\
&=\frac{1}{2}\|\mathcal A\|^2-\langle\operatorname{Diag}{\big(((U^{(1)})^\tp ,\dots,(U^{(k)})^\tp )}\cdot\mathcal A\big),\mathbf x\rangle+\frac{1}{2}\|\mathbf x\|^2\\
&=\frac{1}{2}(\|\mathcal A\|^2-\|\mathbf x\|^2).
\end{align*}
By {\eqref{eq:gradient-x}} and the definition of $\mathbf x$, we also have
\[
\operatorname{grad}_{\mathbf x}g(\mathbb U,\mathbf x)=\mathbf x-\operatorname{Diag}\big(((U^{(1)})^\tp ,\dots,(U^{(k)})^\tp )\cdot\mathcal A\big)=\mathbf 0.
\]
Since
\[
\operatorname{grad}_{U^{(i)}}g(\mathbb U,\mathbf x)=(I-\frac{1}{2}U^{(i)}(U^{(i)})^\tp )(V^{(i)}\Gamma-U^{(i)}(V^{(i)}\Gamma)^\tp  U^{(i)}),\ \text{for all }i=1,\dots,k,
\]
and
\[
\nabla_i f(\mathbb U)=V^{(i)}\Gamma,\ \text{for all }i=1,\dots,k,
\]
where $\Gamma=\operatorname{diag}(\mathbf x)$ is the diagonal matrix formed by the vector $\mathbf x$, {the assertion will follow} if we can show that
\[
\|I-\frac{1}{2}U^{(i)}(U^{(i)})^\tp \|_F\leq \kappa_1
\]
is uniformly bounded by $\kappa_1>0$ over $\|\mathbb U-\mathbb U^*\|_F\leq \epsilon$. This is obviously true. The proof is then complete.
\end{proof}

\section{iAPD Algorithm and Convergence Analysis}\label{sec:APPDT}

\subsection{Description of iAPD algorithm}\label{sec:algorithm}
In \cite{GC-19}, an alternating polar decomposition (APD) algorithm is proposed to solve the optimization problem \eqref{eq:sota-max}. {Algorithm~\ref{algo} below is an improved version of the classical APD, which we call  iAPD.} It is an alternating polar decomposition method with adaptive proximal corrections and truncations. An iteration step in iAPD with a truncation is called a \textit{truncation iteration}. Obviously, there are at most $r$ truncation iterations.

\begin{algorithm}
{iAPD: Low Rank Orthogonal Tensor Approximation}\label{algo}\\
\begin{boxedminipage}{1\textwidth}
Input: a nonzero tensor $\mathcal A\in\mathbb R^{n_1}\otimes\dots\otimes\mathbb R^{n_k}$, a positive integer $r$, and a proximal parameter $\epsilon$.
\begin{algorithmic}
\STATE{\textbf{Step 0} [Initialization]: choose $\mathbb U_{[0]}:=(U^{(1)}_{[0]},\dots,U^{(k)}_{[0]})\in V(r,n_1)\times\dots\times V(r,n_k)$ such that
$f(\mathbb U_{[0]})>0$, and a truncation parameter $\kappa\in(0,\sqrt{f(\mathbb U_{[0]})/r})$. Let $p:=1$.}

\medskip

\STATE{\textbf{Step 1} [Alternating Polar Decompositions-APD]:
Let $i:=1$.  }

\medskip

\STATE{Substep 1 [Polar Decomposition]: If $i>k$, go to Step 2. Otherwise, for all $j=1,\dots,r$, let
\begin{equation}\label{eq:inter-vector}
\mathbf x^{i}_{j,[p]}:=(\mathbf u^{(1)}_{j,[p]},\dots,\mathbf u^{(i-1)}_{j,[p]}, \mathbf u^{(i)}_{j,[p-1]}, \mathbf u^{(i+1)}_{j,[p-1]},\dots,\mathbf u^{(k)}_{j,[p-1]}),
\end{equation}
where $\mathbf u^{(i)}_{j,[p]}$ is the {$j$-th} column of the factor matrix $U^{(i)}_{[p]}$.

Compute the matrix $\Lambda^{(i)}_{[p]}$ as
\begin{equation}\label{eq:lambda-orth}
\Lambda^{(i)}_{[p]}:=\operatorname{diag}(\lambda^{i-1}_{1,[p]},\dots,\lambda^{i-1}_{r,[p]})\ \text{with }\lambda^{i-1}_{j,[p]}:=\mathcal A\tau(\mathbf x^{i}_{j,[p]})\ \text{for }j=1,\dots,r,
\end{equation}
and the matrix $V^{(i)}_{[p]}$ as
\begin{equation}\label{eq:marixv}
V^{(i)}_{[p]}:=\begin{bmatrix}\mathbf v^{(i)}_{1,[p]}&\dots&\mathbf v^{(i)}_{r,[p]}\end{bmatrix}\ \text{with } \mathbf v^{(i)}_{j,[p]}:=\mathcal A\tau_i(\mathbf x^{i}_{j,[p]})\ \text{for }j=1,\dots,r.
\end{equation}

Compute the singular value decomposition of the matrix $V^{(i)}_{[p]}\Lambda^{(i)}_{[p]}$ as
\begin{equation}\label{eq:svd}
V^{(i)}_{[p]}\Lambda^{(i)}_{[p]} =G^{(i)}_{[p]}\Sigma^{(i)}_{[p]}(H^{(i)}_{[p]})^\tp ,\ G^{(i)}_{[p]}\in V(r,n_i)\ \text{and }H^{(i)}_{[p]}\in \O(r),
\end{equation}
where the singular values $\sigma^{(i)}_{j,[p]}$'s are ordered nonincreasingly in the diagonal matrix $\Sigma^{(i)}_{[p]}$. Then the polar decomposition of the matrix $V^{(i)}_{[p]}\Lambda^{(i)}_{[p]}$ is
\begin{equation}\label{eq:polar}
V^{(i)}_{[p]}\Lambda^{(i)}_{[p]}=U^{(i)}_{[p]}S^{(i)}_{[p]}\ \text{with }U^{(i)}_{[p]}:=G^{(i)}_{[p]}(H^{(i)}_{[p]})^\tp,\ {S^{(i)}_{[p]} \coloneqq H^{(i)}_{[p]} \Sigma^{(i)}_{[p]} (H^{(i)}_{[p]})^\tp.}
\end{equation}
}

\medskip

\STATE{Substep 2 [Proximal Correction]: If $\sigma^{(i)}_{r,[p]}<\epsilon$, then compute the polar decomposition of the matrix $V^{(i)}_{[p]}\Lambda^{(i)}_{[p]}+\epsilon U^{(i)}_{[p-1]}$ as
\begin{equation}\label{eq:proximal-polar}
V^{(i)}_{[p]}\Lambda^{(i)}_{[p]}+\epsilon U^{(i)}_{[p-1]}=\hat U^{(i)}_{[p]}\hat S^{(i)}_{[p]}
\end{equation}
for an orthonormal matrix $\hat U^{(i)}_{[p]}\in V(r,n_i)$ and a symmetric positive semidefinite matrix $\hat S^{(i)}_{[p]}$.
Update $U^{(i)}_{[p]}:=\hat U^{(i)}_{[p]}$, and $S^{(i)}_{[p]}:=\hat S^{(i)}_{[p]}$. Set $i:=i+1$ and go to Substep 1.}

\medskip

\STATE{\textbf{Step 2} [Truncation]: If $\lambda^0_{j,[p+1]}=\big((U^{(i)}_{[p]})^\tp V^{(i)}_{[p]}\big)_{jj}<\kappa$ for {some $j\in J\subseteq\{1,\dots,r\}$}, replace the matrices $U^{(i)}_{[p]}$'s by $\hat U^{(i)}_{[p]}$'s, where $\hat U^{(i)}_{[p]}$ is an $n_i\times (r-|J|)$ matrix formed by the columns of $U^{(i)}_{[p]}$ corresponding to $\{1,\dots,r\}\setminus J$, for all $i\in\{1,\dots,k\}$. Update $r:=r-|J|$, and $U^{(i)}_{[p]}:=\hat U^{(i)}_{[p]}$ for all $i\in\{1,\dots,k\}$. Go to Step 3.}

\medskip
\STATE{\textbf{Step 3}: Unless a termination criterion is satisfied, let $p:=p+1$ and go back to Step 1.}
\end{algorithmic}
\end{boxedminipage}
\end{algorithm}
\subsection{Properties of iAPD}\label{sec:properties}
In this section, we derive some inequalities for the increments of the objective function during iterations. {To do this, we define}
\begin{equation}\label{eq:variable-u}
\mathbb U_{i,[p]} \coloneqq (U^{(1)}_{[p]},\dots,U^{(i)}_{[p]},U^{(i+1)}_{[p-1]},\dots,U^{(k)}_{[p-1]})
\end{equation}
for each $1\le i \le k$, and
\begin{equation}\label{eq:whole-u}
\mathbb U_{[p]} \coloneqq (U^{(1)}_{[p]},\dots,U^{(k)}_{[p]}),
\end{equation}
{which is equal to $\mathbb U_{k,[p]} =\mathbb U_{0,[p+1]}$} for each $p \in \mathbb{N}$. We remark that the $j$-th column of a factor matrix $U^{(i)}_{[p]}$ is denoted by $\mathbf u^{(i)}_{j,[p]}$ for each $1\le j \le r$ while the superscript $i$ for the \textit{block vector} $\mathbf x^{i}_{j,[p]}$ is not bracketed. For each $1\le j \le r$ and $p\in \mathbb{N}$, {we also denote
\begin{equation}\label{eq:lambda-notation}
\lambda^k_{j,[p-1]} = \lambda^0_{j,[p]},
\end{equation}
where $\lambda^{i-1}_{j,[p]}$ is defined in \eqref{eq:lambda-orth} for the $i$-th iteration.}
One immediate observation is that if the {$p$-th iteration in Algorithm~\ref{algo}} is not a truncation iteration, then the sizes of the matrices in $\mathbb U_{[p]}$ and those in $\mathbb U_{[p-1]}$ are the same. Also if the number of iterations {in Algorithm~\ref{algo}} is infinite, then there is a sufficiently large $N_0$ such that the {$p$-th iteration} is not a truncation iteration for any $p\geq N_0$. The proof of the next lemma can be found in Appendix~\ref{app:proofalgo-lambadk}.

\begin{lemma}[Monotonicity of iAPD]\label{lem:lambda-k}
If the $p$-th iteration in Algorithm~\ref{algo} is not a truncation iteration, then for each $0\le i \le k-1$, we have
\begin{equation}\label{eq:obj-k}
f(\mathbb U_{i+1,[p]})-f(\mathbb U_{i,[p]})\geq \frac{\epsilon}{2}\|U^{(i+1)}_{[p]}-U^{(i+1)}_{[p-1]}\|_F^2.
\end{equation}
\end{lemma}

\begin{proposition}[Sufficient Decrease]\label{prop:subfficient}
{If the $p$-th iteration in Algorithm~\ref{algo} is not a truncation iteration, then we have}
\begin{equation}\label{eq:obj-suf}
f(\mathbb U_{[p]})-f(\mathbb U_{[p-1]})\geq \frac{\epsilon}{2}\|\mathbb U_{[p]}-\mathbb U_{[p-1]}\|_F^2.
\end{equation}
\end{proposition}

\begin{proof}
{We have
\[
f(\mathbb U_{[p]})-f(\mathbb U_{[p-1]})=\sum_{i=0}^{k-1}\big(f(\mathbb U_{i+1,[p]})-f(\mathbb U_{i,[p]})\big)\geq \frac{\epsilon}{2}\sum_{i=0}^{k-1}\|U^{(i+1)}_{[p]}-U^{(i+1)}_{[p-1]}\|_F^2=\frac{\epsilon}{2}\|\mathbb U_{[p]}-\mathbb U_{[p-1]}\|_F^2,
\]
where the inequality follows from \eqref{eq:obj-k} in Lemma~\ref{lem:lambda-k}.}
\end{proof}

At each truncation iteration, the number of columns of the matrices in $\mathbb U_{[p]}$ is decreased strictly. The first issue we have to address is that the iteration $\mathbb U_{[p]}$ is not vacuous, i.e., the numbers of the columns of the matrices in $\mathbb U_{[p]}$ {are stable and positive}. We have the following proposition, which is recorded for latter reference.
\begin{proposition}\label{prop:positive}
The number of columns of $U^{(i)}_{[p]}$'s will be stable at { a positive integer $s\leq r$} 
and there exists $N_0$ such that $f(\mathbb U_{[p]})$ is nondecreasing for all $p\geq N_0$.
\end{proposition}

\begin{proof}
Since the initial number $r$ of columns is finite, the truncation will occur at most $r$ times and the total decreased number of columns of matrices in $\mathbb U_{[p]}$ is {bounded above} by $r$.
It follows from Step 2 of Algorithm~\ref{algo} and Lemma~\ref{lem:lambda-k} that if {the $p$-th iteration} is a truncation iteration and the number of columns of the matrices in $\mathbb U_{[p-1]}$ is decreased from $r_1$ to $r_2<r_1$, then we have
\[
f(\mathbb U_{[p]})\geq f(\mathbb U_{[p-1]})-(r_1-r_2)\kappa^2.
\]
By the truncation strategy in Algorithm~\ref{algo}, {after all the truncation iterations}, {the value of the objective function will decrease} at most $r\kappa^2$. {Moreover, at each iteration without truncation, the value of the objective function} is nondecreasing by Lemma~\ref{lem:lambda-k} and $r\kappa^2<f(\mathbb U_{[0]})$. Hence, {$\mathbb U_{[p]}$ cannot be vacuous.} As there can only be a finite number of truncations, there exists $N_0$ such that for any $p\geq N_0$, the $p$-th iteration is not a truncation iteration, {and the conclusion then follows.}
\end{proof}

Let us consider the following optimization problem
\begin{equation}\label{eq:sota-unconstrained}
\max_{\mathbb U} h(\mathbb U):=f(\mathbb U)+\sum_{i=1}^k\delta_{V(r,n_i)}(U^{(i)}).
\end{equation}
{It is straightforward to verify that} \eqref{eq:sota-unconstrained} is an unconstrained reformulation of problem \eqref{eq:sota-max} and $h$ is a KL function according to Lemma~\ref{lem:KL functions}. Readers can find the proof of the next lemma in Appendix~\ref{app:proofalgo-sub}.
\begin{lemma}[Subdifferential Bound]\label{lem:subdiff}
If the $(p+1)$-st iteration is not a truncation iteration, then there exists a subdifferential $\mathbb W_{[p+1]}\in\partial h(\mathbb U_{[p+1]})$ such that
\begin{equation}\label{eq:subdiff}
\|\mathbb W_{[p+1]}\|_F\leq \sqrt{k}(2r\sqrt{r}{\|\mathcal A\|^2} +\epsilon)\|\mathbb U_{[p+1]}-\mathbb U_{[p]}\|_F.
\end{equation}
\end{lemma}

\subsection{Global convergence}\label{sec:global}

The following classical result can be found in \cite{ABS-13}.
\begin{lemma}[Abstract Convergence]\label{lem:abs-conv}
Let $p : \mathbb R^n\rightarrow\mathbb R\cup\{\pm\infty\}$ be a proper lower semicontinuous function and {let $\{\mathbf x^{(k)}\}\subseteq \mathbb R^n$ be} a sequence satisfying the following properties
\begin{enumerate}
\item there is a constant $\alpha>0$ such that
\[
p(\mathbf x^{(k+1)})-p(\mathbf x^{(k)})\geq \alpha\|\mathbf x^{(k+1)}-\mathbf x^{(k)}\|^2,
\]
\item there is a constant $\beta>0$ and a $\mathbf w^{(k+1)}\in \partial p(\mathbf x^{(k+1)})$ such that
\[
\|\mathbf w^{(k+1)}\|\leq \beta\|\mathbf x^{(k+1)}-\mathbf x^{(k)}\|,
\]
\item there is a subsequence $\{\mathbf x^{(k_i)}\}$ of $\{\mathbf x^{(k)}\}$ and $\mathbf x^*\in\mathbb R^n$ such that
\[
\mathbf x^{(k_i)}\rightarrow\mathbf x^*\ \text{and }p(\mathbf x^{(k_i)})\rightarrow p(\mathbf x^*)\ \text{as }i\rightarrow\infty.
\]
\end{enumerate}
If $p$ has the Kurdyka-\L ojasiewicz property at the point $\mathbf x^*$, then the whole sequence $\{\mathbf x^{(k)}\}$ converges to $\mathbf x^*$, and $\mathbf x^*$ is a critical point of $p$.
\end{lemma}

Regarding the global convergence of Algorithm~\ref{algo}, by Proposition~\ref{prop:positive}, we can assume without loss of generality that there is no truncation iteration in the sequence $\{\mathbb U_{[p]}\}$ generated by Algorithm~\ref{algo}.
\begin{proposition}\label{prop:monotone}
Given a sequence $\{\mathbb U_{[p]}\}$ generated by Algorithm~\ref{algo}, the sequence $\{f(\mathbb U_{[p]})\}$ increases monotonically and hence converges.
\end{proposition}

\begin{proof}
Since the sequence $\{\mathbb U_{[p]}\}$ is bounded, $\{f(\mathbb U_{[p]})\}$ is bounded as well by the definition (cf.\ \eqref{eq:objective}). The {convergence} then follows from Proposition~\ref{prop:positive}.
\end{proof}
\begin{theorem}[Global Convergence]\label{thm:global}
{Any sequence $\{\mathbb U_{[p]}\}$ generated by Algorithm~\ref{algo} is bounded and converges to a KKT point of the problem~\eqref{eq:sota-max}.}
\end{theorem}

\begin{proof}
Obviously, the sequence $\{\mathbb U_{[p]}\}$ is bounded and the function $h$ is continuous on the product of the Stiefel manifolds. The {convergence} follows from Proposition~\ref{prop:subfficient}, Lemma~\ref{lem:subdiff}, Lemma~\ref{lem:abs-conv}, and Proposition~\ref{prop:monotone}.
\end{proof}

\subsection{{Sublinear convergence rate}}\label{sec:sublinear}
We consider the following function
\begin{equation}\label{eq:function-aml}
q(\mathbb U,\mathbb P) \coloneqq f(\mathbb U)-\sum_{i=1}^k\langle P^{(i)},(U^{(i)})^\tp U^{(i)}-I\rangle,
\end{equation}
{which is a polynomial of degree $2k$ in $N \coloneqq  \sum_{i=1}^k (r n_i + \binom{r+1}{2})$ variables:
\[
(\mathbb U,\mathbb P) = (U^{(1)},\dots,U^{(k)},P^{(1)},\dots, P^{(k)})\in \mathbb{R}^{n_1 \times r} \times\dots\times \mathbb{R}^{n_k \times r}\times \operatorname{S}^{r\times r}\times\dots\times\operatorname{S}^{r\times r}.
\]
Let
\begin{equation}\label{eq:tau}
\tau \coloneqq 1-\frac{1}{2k(6k-3)^{N-1}},
\end{equation}
which is the \textit{\L ojasiewicz exponent} of the polynomial $q$ obtained by Lemma~\ref{lemma:loja1}. {We suppose that $\mathbb U^*$ is a KKT point of \eqref{eq:sota-max} with the multiplier $\mathbb P^*$.}
{For}
\begin{equation}\label{eq:axl-fn}
\hat q(\mathbb U,\mathbb P) \coloneqq q(\mathbb U,\mathbb P)-q(\mathbb U^*,\mathbb P^*),
\end{equation}
we must have
\[
\hat q(\mathbb U^*,\mathbb P^*)=0,\quad \nabla \hat q(\mathbb U^*,\mathbb P^*)=0.
\]
Thus {according to Lemma~\ref{lemma:loja1}}, there exist some $\gamma,c>0$ such that
\[
\|\nabla \hat q(\mathbb U,\mathbb P)\|_F\geq c|\hat q(\mathbb U,\mathbb P)|^{\tau}\ \text{whenever }\|(\mathbb U,\mathbb P)-(\mathbb U^*,\mathbb P^*)\|_F\leq \gamma.
\]
Therefore,
\begin{equation}\label{eq:loj-ineq}
\sum_{i=1}^k\|\nabla_i f(\mathbb U)-2U^{(i)}P^{(i)}\|_F^2\geq c^2(f(\mathbb U)-f(\mathbb U^*))^{2\tau}
\end{equation}
for {any feasible point $\mathbb U$ of \eqref{eq:sota-max} satisfying $\|(\mathbb U,\mathbb P)-(\mathbb U^*,\mathbb P^*)\|_F\leq \gamma$.}
\begin{theorem}[Sublinear Convergence Rate]\label{thm:sublinear}
Let $\{\mathbb U_{[p]}\}$ be a sequence generated by Algorithm~\ref{algo} for a given nonzero tensor $\mathcal A\in\mathbb R^{n_1}\otimes\dots\otimes\mathbb R^{n_k}$ and let $\tau$ be defined by \eqref{eq:tau}.  The following statements hold:
\begin{itemize}
\item[{\rm (1)}] {the sequence} $\{f(\mathbb U_{[p]})\}$ converges to $f^*$, with sublinear convergence rate at least
$O(p^{\frac{1}{1-2\tau}})$, that is, there exist $M_1>0$ and $p_0 \in \mathbb{N}$ such that for all $p \geq p_0$
\begin{equation}\label{eq:sub-linear}
f^*-f(\mathbb U_{[p]}) \leq M_1 \, p^{\frac{1}{1-2\tau}};
\end{equation}
\item[{\rm (2)}] $\{\mathbb U_{[p]}\}$ converges to $\mathbb U^*$ globally with the sublinear convergence rate at least
{$O(p^{\frac{\tau-1}{2\tau-1}})$}, that is, there exist $M_2>0$ and $p_0 \in \mathbb{N}$ such that for all $p \ge p_0$
\[
\|\mathbb U_{[p]} - \mathbb U^*\|_F \leq M_2 \, p^{\frac{\tau-1}{2\tau-1}}.
\]
\end{itemize}
\end{theorem}

\begin{proof}
In the following, we consider the sequence $\{\mathbb U_{[p]}\}$ generated by Algorithm~\ref{algo}.
Let
\[
P^{(i)}_{[p]}:=S^{(i)}_{[p]}-\alpha_{i,[p]} I:=\begin{cases}S^{(i)}_{[p]}-\epsilon I&\text{if proximal correction is executed},\\ S^{(i)}_{[p]}&\text{otherwise},\end{cases}
\]
where $\alpha_{i,[p]}\in\{0,\epsilon\}$.
We also have that
\[
S^{(i)}_{[p]}=\begin{cases}(U^{(i)}_{[p]})^\tp (V^{(i)}_{[p]}\Lambda^{(i)}_{[p]}+\epsilon U^{(i)}_{[p-1]})&\text{if proximal correction is executed},\\
(U^{(i)}_{[p]})^\tp V^{(i)}_{[p]}\Lambda^{(i)}_{[p]}&\text{otherwise}. \end{cases}
\]
Note that $\{\mathbb U_{[p]}\}$ converges by Theorem~\ref{thm:global} and hence {$\{ V^{(i)}_{[p]}\Lambda^{(i)}_{[p]} \}$ converges}. Recall that the proximal correction step is determined by singular values of the matrices $V^{(i)}_{[p]}\Lambda^{(i)}_{[p]}$'s. Thus, for sufficiently large $p$ (say $p\geq p_0$), $\alpha_{i,[p]}$ will be stable for all $p$ and $1 \le i \le r$.
By Lemma~\ref{lem:subdiff} and \eqref{eq:w}, we have
\begin{equation}\label{eq:sub-sublinear}
\|\nabla_i f(\mathbb U_{[p+1]})-2U^{(i)}_{[p+1]}S^{(i)}_{[p+1]}+2\alpha U^{(i)}_{[p+1]}\|_F\leq (r\sqrt{r}{\|\mathcal A\|^2} +\epsilon)\|U^{(i)}_{[p+1]}-U^{(i)}_{[p]}\|_F.
\end{equation}
Since $\{\mathbb U_{[p]}\}$ converges by Theorem~\ref{thm:global}, we see that
\[
\lim_{p\rightarrow\infty}P^{(i)}_{[p]}=\lim_{p\rightarrow\infty}(U^{(i)}_{[p]})^\tp \big(V^{(i)}_{[p]}\Lambda^{(i)}_{[p]}+\alpha U^{(i)}_{[p-1]}\big)-\alpha I=(U^{(*,i)})^\tp V^{(*,i)}\Lambda^*=P^{(*,i)}.
\]
Hence for sufficiently large $p$, we may conclude that
\[
\|(\mathbb U_{[p]},\mathbb P_{[p]})-(\mathbb U^*,\mathbb P^*)\|_F\leq \gamma.
\]
{
This implies }
\begin{align}
c^2(f(\mathbb U_{[p]})-f(\mathbb U^*))^{2\tau}\nonumber &\leq \sum_{i=1}^k\|\nabla_i f(\mathbb U_{[p]})-2U_{[p]}^{(i)}P_{[p]}^{(i)}\|_F^2\nonumber\\
&\leq 2\sum_{i=1}^k\|\nabla_i f(\mathbb U_{[p+1]})-2U_{[p+1]}^{(i)}P_{[p+1]}^{(i)}\|_F^2\nonumber\\
& +2\sum_{i=1}^k\|\nabla_i f(\mathbb U_{[p]})-2U_{[p]}^{(i)}P_{[p]}^{(i)}-\big(\nabla_i f(\mathbb U_{[p+1]})-2U_{[p+1]}^{(i)}P_{[p+1]}^{(i)}\big)\|_F^2\label{eq:fun-lip}\\
&\leq 2(k(r\sqrt{r}{\|\mathcal A\|^2} +\epsilon)^2+L)\|\mathbb U_{[p+1]}-\mathbb U_{[p]}\|_F^2\nonumber\\
&\leq (k(r\sqrt{r}{\|\mathcal A\|^2} +\epsilon)^2+L)\epsilon(f(\mathbb U_{[p+1]})-f(\mathbb U_{[p]}),\nonumber
\end{align}
where the first inequality follows from \eqref{eq:loj-ineq}, the third from \eqref{eq:sub-sublinear} and {the fact that} the function in \eqref{eq:fun-lip} is Lipschitz continuous since $\alpha_{i,[p]}$ is stable for sufficiently large $p$, and the last one follows from Proposition~\ref{prop:subfficient}. Here $L$ is the Lipschitz constant of the function in \eqref{eq:fun-lip} {on} the product of Stiefel manifolds.

{If we set $\beta_p \coloneqq f(\mathbb U^*)-f(\mathbb U_{[p]})$, then we have}
\[
\beta_p-\beta_{p+1}\geq M\beta_p^{2\tau}
\]
for some constant $M>0$,
from which we can show
\[
\beta^{1-2\tau}_{p+1}-\beta_p^{1-2\tau}\geq (2\tau-1)M.
\]
Thus,
\[
\beta^{1-2\tau}_{p}\geq M(2\tau-1)(p-p_0)+\beta_{p_0}^{1-2\tau}
\]
and the conclusion follows {since} $\tau\in(0,1)$.
For a more detailed analysis on the sequence $\{\beta_p\}$, we refer readers to \cite[Section~3.4]{HL-18}.
\end{proof}

{We remark} that the convergence rate in \eqref{eq:sub-linear} is faster than the classical $O(1/p)$ for first order methods in optimization \cite{Beck-book}, while the optimal rate is $O(1/p^2)$ for convex problems by the celebrated work of Nesterov \cite{N-04}.
\subsection{{Linear convergence}}\label{sec:linear}
In this section, we will establish the linear convergence of Algorithm~\ref{algo}. The proof of the next lemma is available in Appendix~\ref{app:proofalgo-gradient}.
\begin{lemma}[Relative Error]\label{lem:gradient-diff}
There exists a constant $\gamma>0$ such that
\[
\|\nabla_i f(\mathbb U_{[p+1]})-U^{(i)}_{[p+1]}(\nabla_i f(\mathbb U_{[p+1]}))^\tp U^{(i)}_{[p+1]}\|_F \leq \gamma \|U^{(i)}_{[p]}-U^{(i)}_{[p+1]}\|_F
\]
for all $1 \le i \le k$ and $p\in \mathbb{N}$.
\end{lemma}

\begin{theorem}[Linear Convergence Rate]\label{thm:linear}
Let $\{\mathbb U_{[p]}\}$ be {a sequence} generated by Algorithm~\ref{algo} for a given nonzero tensor $\mathcal A\in\mathbb R^{n_1}\otimes\dots\otimes\mathbb R^{n_k}$.
If $\mathbb U_{[p]}\rightarrow \mathbb U^*$ with $\mathbb U^*$ a nondegenerate KKT point of \eqref{eq:sota-max}, then the whole sequence $\{\mathbb U_{[p]}\}$ converges $R$-linearly to $\mathbb U^*$.
\end{theorem}

\begin{proof}
By Theorem~\ref{thm:global}, the sequence $\{\mathbb U_{[p]}\}$ converges globally to $\mathbb U^*$, which together with $\mathbf x^*:=\operatorname{Diag}\big(((U^{(*,1)})^\tp ,\dots,(U^{(*,k)})^\tp )\cdot\mathcal A\big)$ is a nondegenerate critical point of the function $g$ on $U_{\mathbf n,r}$. Hence for a sufficiently large $p$, Lemma~\ref{lem:gradient} implies that
\[
\sum_{i=1}^k\|\nabla_i f(\mathbb U_{[p]})-U^{(i)}_{[p]}\nabla_i (f(\mathbb U_{[p]}))^\tp U^{(i)}_{[p]}\|_F^2\geq \kappa |f(\mathbb U_{[p]})-f(\mathbb U^*)|.
\]
On the other hand, by Lemma~\ref{lem:gradient-diff}, we have
\[
\sum_{i=1}^k\|\nabla_i f(\mathbb U_{[p]})-U^{(i)}_{[p]}\nabla_i (f(\mathbb U_{[p]}))^\tp U^{(i)}_{[p]}\|_F^2\leq k\gamma^2\|\mathbb U_{[p]}-\mathbb U_{[p-1]}\|_F^2.
\]
Thus,
\begin{align*}
f(\mathbb U_{[p]})-f(\mathbb U_{[p-1]})&\geq \frac{\epsilon}{2}\|\mathbb U_{[p]}-\mathbb U_{[p-1]}\|_F^2\\
&\geq \frac{\kappa\epsilon}{2k\gamma^2}(f(\mathbb U^*)-f(\mathbb U_{[p]})),
\end{align*}
where the first inequality follows from Proposition~\ref{prop:subfficient}, and the second follows from the preceding two inequalities and Proposition~\ref{prop:monotone}.
Therefore, for a sufficiently large $p$, we have
\begin{equation}\label{eq:linear-obj}
f(\mathbb U^*)-f(\mathbb U_{[p]})\leq \frac{2k\gamma^2}{2k\gamma^2+\kappa\epsilon}\big(f(\mathbb U^*)-f(\mathbb U_{[p-1]})\big),
\end{equation}
which establishes the local $Q$-linear convergence of the sequence $\lbrace f(\mathbb{U}_{\left[ p \right]}) \rbrace$. Consequently, we have
\begin{align*}
\|\mathbb U_{[p]}-\mathbb U_{[p-1]}\|_F&\leq\sqrt{\frac{2}{\epsilon}}\sqrt{f(\mathbb U_{[p]})-f(\mathbb U_{[p-1]})}\\
&\leq \sqrt{\frac{2}{\epsilon}}\sqrt{f(\mathbb U^*)-f(\mathbb U_{[p-1]})}\\
&\leq \sqrt{\frac{2}{\epsilon}}\Bigg[\sqrt{\frac{2k\gamma^2}{2k\gamma^2+\kappa\epsilon}}\Bigg]^{p-1}\sqrt{f(\mathbb U^*)-f(\mathbb U_{[0]})},
\end{align*}
which implies that
\[
\sum_{p=p_0}^\infty\|\mathbb U_{[p]}-\mathbb U_{[p-1]}\|_F<\infty
\]
{for any sufficiently large positive integer $p_0$.}
As $\mathbb U_{[p]}\rightarrow\mathbb U^*$, we have
\[
\|\mathbb U_{[p]}-\mathbb U^*\|_F\leq \sum_{s=p}^\infty\|\mathbb U_{[s+1]}-\mathbb U_{[s]}\|_F.
\]
{Hence, we obtain}
\[
\|\mathbb U_{[p]}-\mathbb U^*\|_F\leq\sqrt{\frac{2}{\epsilon}}\sqrt{f(\mathbb U^*)-f(\mathbb U_{[0]})}\frac{1}{1-\sqrt{\frac{2k\gamma^2}{2k\gamma^2+\kappa\epsilon}}}\Big[\sqrt{\frac{2k\gamma^2}{2k\gamma^2+\kappa\epsilon}}\Big]^{p},
\]
which is the claimed $R$-linear convergence of the sequence $\{\mathbb U_{[p]}\}$ {and this completes the proof.}
\end{proof}

The following result follows from Theorems~\ref{thm:linear} and \ref{thm:kkt}.
\begin{theorem}[Generic Linear Convergence]\label{thm:generic}
If $\{\mathbb U_{[p]}\}$ {is a sequence generated} by Algorithm~\ref{algo} for a generic tensor $\mathcal A\in\mathbb R^{n_1}\otimes\dots\otimes\mathbb R^{n_k}$, then the sequence $\{\mathbb U_{[p]}\}$ converges $R$-linearly to a KKT point of \eqref{eq:sota-max}.
\end{theorem}
\section{Discussions on Proximal Correction and Truncation}\label{sec:discussion}
In this section, we will {carry out} a further {study of} proximal corrections and {truncation iterations} in Algorithm~\ref{algo}. We will prove that if {we make an appropriate assumption on the limiting point}, then the truncation iteration is unnecessary and proximal corrections {are only needed for finitely many times.} Thus, {our iAPD} reduces to {the classical APD} proposed in \cite{GC-19} {after finitely many iterations}. Remarkably, the assumption on the whole {iteration sequence} (cf.\ \cite[Assumption~A]{GC-19}) is {vastly relaxed} to a requirement on the {limiting point}. {Together with conclusions in Section~\ref{sec:manifold} about KKT points, our results in this section can shed some light on the further understanding of APD and iAPD.}

\subsection{{Proximal correction}}\label{sec:proximal}
In this subsection, we will prove that in most situations, the proximal correction in Algorithm~\ref{algo} is unnecessary. {Before we proceed, we introduce the notion of \textit{regular KKT point}}.

\begin{definition}[Regular KKT Point]\label{def:regular}
A KKT point $\mathbb U:=(U^{(1)},\dots,U^{(k)})\in\mathbb R^{n_1\times r}\times\dots\times\mathbb R^{n_k\times r}$ of \eqref{eq:sota-max} is called a \textit{regular KKT point} if the matrix $V^{(i)}\Lambda$ (cf.\ \eqref{eq:critial} and \eqref{eq:lambda-tensor})
is of rank at least $\min\{r,n_i-1\}$ for each $1\le i \le k$.
\end{definition}

The requirement for a regular KKT point in Definition~\ref{def:regular} is natural. {The matrix $U^{(i)}$ is orthonormal and hence it is of full rank, for each $1 \le i \le k$. On the other hand, these matrices are polar orthonormal factor matrices of $V^{(i)}\Lambda$'s by Algorithm~\ref{algo}.} {If for some $1 \le i \le k$, the matrix $V^{(i)}\Lambda$ is of defective rank}, then the best rank $r$ approximation {of} the $i$-th factor matrix is not unique. The case that $r=n_i$ is an exceptional case, see Lemma~\ref{lem:error-def}.
With Lemma~\ref{lem:error-def}, we have a revised proximal correction step, {{which is} described in Algorithm~\ref{algo:revised}. }


\begin{algorithm}{Revisited Proximal Step}\label{algo:revised}\\
\begin{boxedminipage}{1\textwidth}
$\tau>\epsilon$ is a given constant.
\begin{algorithmic}

\STATE{Substep 2 [Revised Proximal Correction]: If $\sigma^{(i)}_{r,[p]}<\epsilon$, then consider the following two cases.
\begin{enumerate}
\item [Case (i)] If $r=n_i$ and $\sigma^{(i)}_{r-1,[p]}\geq\tau$, then define a vector
    \[
    \hat{\mathbf g}^{(i)}_r:=\begin{cases}-\mathbf g^{(i)}_r&\text{if }\langle \mathbf g^{(i)}_r,(U^{(i)}_{[p-1]}H^{(i)}_{[p]})_r\rangle<0\\ \mathbf g^{(i)}_r&\text{otherwise}\end{cases}
    \]
    where $\mathbf g^{(i)}_r$ is the {$r$-th} column of the matrix $G^{(i)}_{[p]}$ and similar for $(U^{(i)}_{[p-1]}H^{(i)}_{[p]})_r$. Form a matrix $\hat G^{(i)}_{[p]}$ from $G^{(i)}_{[p]}$ by replacing the last column with $\hat{\mathbf g}^{(i)}_r$. Let
\begin{equation}\label{eq:revised-polar}
U^{(i)}_{[p]}:=\hat G^{(i)}_{[p]}(H^{(i)}_{[p]})^\tp \ \text{and }S^{(i)}_{[p]}:=(U^{(i)}_{[p]})^\tp V^{(i)}_{[p]}\Lambda^{(i)}_{[p]}.
\end{equation}

\item [Case (ii)]
For the other cases, compute the polar decomposition of the matrix $V^{(i)}_{[p]}\Lambda^{(i)}_{[p]}+\epsilon U^{(i)}_{[p-1]}$ as
\begin{equation*}
V^{(i)}_{[p]}\Lambda^{(i)}_{[p]}+\epsilon U^{(i)}_{[p-1]}=\hat U^{(i)}_{[p]}\hat S^{(i)}_{[p]}
\end{equation*}
for an orthonormal matrix $\hat U^{(i)}_{[p]}\in V(r,n_i)$ and a symmetric positive semidefinite matrix $\hat S^{(i)}_{[p]}$.
Update $U^{(i)}_{[p]}:=\hat U^{(i)}_{[p]}$ and $S^{(i)}_{[p]}:=\hat S^{(i)}_{[p]}$.
\end{enumerate}

Set $i:=i+1$ and go to Substep 1.}
\end{algorithmic}
\end{boxedminipage}
\end{algorithm}

If we replace the Substep 2 in Algorithm~\ref{algo} by the revised version described in Algorithm~\ref{algo:revised}, then the sequence $\{\mathbb U_{[p]}\}$ still has the sufficient {decreasing property.} This is the content of the following lemma.

\begin{lemma}[Revised Version]\label{lem:relative-decrease}
Suppose that $\tau>\epsilon {> 0}$.
For {any $p \in \mathbb{N}$ such that the $p$-th iteration} is not a truncation iteration, we have
\begin{equation}\label{eq:obj-k-rev}
f(\mathbb U_{i+1,[p]})-f(\mathbb U_{i,[p]})\geq \frac{1}{2}\min\{\epsilon,\tau-\epsilon\}\|U^{(i+1)}_{[p]}-U^{(i+1)}_{[p-1]}\|_F^2,\quad 0 \le i \le k-1.
\end{equation}
Moreover, {the matrix $S^{(i)}_{[p]}$ defined in \eqref{eq:revised-polar} is symmetric}.
\end{lemma}

\begin{proof}
The matrix $S^{(i)}_{[p]}$ is obviously symmetric by {a direct calculation. For the decreasing property, it is sufficient to prove the result for Case (i) in Algorithm~\ref{algo:revised}.}
{We} suppose that at iteration $p$ and $i$, Case (i) of Algorithm~\ref{algo:revised} is executed. {In this case,} $\mathbf g^{(i)}_r$ is totally determined (up to sign) by the first $(r-1)$ columns of the matrix $G^{(i)}_{[p]}$. It follows from Lemma~\ref{lem:error-def} and the choice of $\hat{\mathbf g}^{(i)}_r$ that
\begin{align}
\|\hat{\mathbf g}^{(i)}_r-(U^{(i)}_{[p-1]}H^{(i)}_{[p]})_r\| &=\min\{\|\mathbf g^{(i)}_r-(U^{(i)}_{[p-1]}H^{(i)}_{[p]})_r\|,\|\mathbf g^{(i)}_r+(U^{(i)}_{[p-1]}H^{(i)}_{[p]})_r\|\} \nonumber \\
&\leq \|(\hat G^{(i)}_{[p]})_1-(U^{(i)}_{[p-1]}H^{(i)}_{[p]})_1\|_F \nonumber \\
& =\|( G^{(i)}_{[p]})_1-(U^{(i)}_{[p-1]}H^{(i)}_{[p]})_1\|_F,\label{eq:revised-rk1}
\end{align}
where $(A)_1$ represents the $n_i\times (r-1)$ submatrix formed by the first $(r-1)$ columns of a given $n_i\times r$ matrix $A$. {Here since $\hat G^{(i)}_{[p]}$ is defined by replacing the last column of $G^{(i)}_{[p]}$ by $\hat{\mathbf g}^{(i)}_r$, we have $(\hat G^{(i)}_{[p]})_1=( G^{(i)}_{[p]})_1$, which immediately implies \eqref{eq:revised-rk1}.}

We then have
\begin{align}
\sum_{j=1}^r\lambda^{i-1}_{j,[p]}(\lambda^{i}_{j,[p]}-\lambda^{i-1}_{j,[p]})
&=\operatorname{Tr}((U^{(i)}_{[p]})^\tp V^{(i)}_{[p]}\Lambda^{(i)}_{[p]})
-\operatorname{Tr}((U^{(i)}_{[p-1]})^\tp V^{(i)}_{[p]}\Lambda^{(i)}_{[p]})\nonumber\\
&=\langle G^{(i)}_{[p]}\Sigma^{(i)}_{[p]}(H^{(i)}_{[p]})^\tp ,U^{(i)}_{[p]}-U^{(i)}_{[p-1]}\rangle\nonumber\\
&=\langle G^{(i)}_{[p]}\Sigma^{(i)}_{[p]},\hat G^{(i)}_{[p]}-U^{(i)}_{[p-1]}H^{(i)}_{[p]}\rangle\nonumber\\
&\geq \langle (G^{(i)}_{[p]})_1\tilde\Sigma^{(i)}_{[p]}, (\hat G^{(i)}_{[p]})_1-(U^{(i)}_{[p-1]}H^{(i)}_{[p]})_1\rangle-\epsilon|\langle \hat{\mathbf g}^{(i)}_r,\hat{\mathbf g}^{(i)}_r-(U^{(i)}_{[p-1]}H^{(i)}_{[p]})_r\rangle|\nonumber\\
&\geq \tau\|(\hat G^{(i)}_{[p]})_1-(U^{(i)}_{[p-1]}H^{(i)}_{[p]})_1\|_F^2-\epsilon\|\hat{\mathbf g}^{(i)}_r-(U^{(i)}_{[p-1]}H^{(i)}_{[p]})_r\|^2\nonumber\\
&\geq (\tau-\epsilon)\|(\hat G^{(i)}_{[p]})_1-(U^{(i)}_{[p-1]}H^{(i)}_{[p]})_1\|_F^2\nonumber\\
&\geq \frac{1}{2}(\tau-\epsilon)\|\hat G^{(i)}_{[p]}-U^{(i)}_{[p-1]}H^{(i)}_{[p]}\|_F^2\nonumber\\
&=\frac{1}{2}(\tau-\epsilon)\|U^{(i)}_{[p]}-U^{(i)}_{[p-1]}\|_F^2,\label{eq:revised-suff}
\end{align}
where $\tilde\Sigma^{(i)}_{[p]}$ is the $(r-1)\times (r-1)$ leading principal submatrix of $\Sigma^{(i)}_{[p]}$, the first inequality follows from $\sigma^{(i)}_{r,[p]}<\epsilon$, the second inequality follows from $\sigma^{(i)}_{r-1,[p]}\geq\tau$, the last two inequalities both follow from \eqref{eq:revised-rk1}.
With \eqref{eq:revised-suff}, {the rest of the proof is the same as that of Lemma~\ref{lem:lambda-k} and the conclusion follows.}
\end{proof}
With Lemma~\ref{lem:relative-decrease}, all the convergence results established in Section~\ref{sec:APPDT} hold as well.

\begin{theorem}[Regular KKT point]\label{thm:regular}
Let $\mathcal A$ be a generic tensor. If $\mathbb U$ is a local maximizer of problem \eqref{eq:sota-max} {where each entry of $\operatorname{Diag}(\Upsilon)$ is nonzero,} then $(\mathbb U, \operatorname{Diag}(\Upsilon))$ is a nondegenerate critical point of {$g$ defined in \eqref{eq:objective-p}} and $\mathbb U$ is a regular KKT point of problem \eqref{eq:sota-max}.
\end{theorem}

\begin{proof}
If $\operatorname{Diag}(\Upsilon)$ is a vector with {all nonzero components}, then {we must have} $(\mathbb U, \operatorname{Diag}(\Upsilon))\in U_{\mathbf n,r}$. Moreover, Proposition~\ref{prop:critical-equivalence} implies that $(\mathbb U, \operatorname{Diag}(\Upsilon))$ is a critical point of $g$ and since $g$ is a Morse function {on} $U_{\mathbf n,r}$ for a generic tensor $\mathcal A$ by Proposition~\ref{prop:nondegnerate}, it is actually a nondegenerate critical point. {According to Lemma~\ref{lem:isolated critical points}, we may conclude that $(\mathbb U, \operatorname{Diag}(\Upsilon))$ is isolated, i.e., $g$ has no other non-degenerate critical point near $(\mathbb U, \operatorname{Diag}(\Upsilon))$.}

In the following, we will {prove} that the matrix $V^{(i)}\Lambda$ defined {by} \eqref{eq:critial} and \eqref{eq:lambda-tensor} is of rank at least $\min\{r,n_i-1\}$ for all $i=1,\dots,k$. {Suppose on the contrary that there exists some $i\in \{1,\dots,k\}$ such that the matrix $V^{(i)}\Lambda$ has rank}
$s<\min\{r,n_i-1\}$. We consider the singular value decomposition of $V^{(i)}\Lambda$
\[
V^{(i)}\Lambda=U\Sigma V^\tp
\]
with $U:=[U_1\ U_2]\in V(r,n_i)$, $U_1\in\mathbb R^{n_i\times s}$, $\Sigma\in\mathbb R^{r\times r}$, $V=[V_1\ V_2]\in \O(r)$ and $V_1\in\mathbb R^{r\times s}$. Hence
\[
V^{(i)}\Lambda=(UV^\tp )(V\Sigma V^\tp )
\]
is a polar decomposition of $V^{(i)}\Lambda$ with the polar orthonormal factor matrix $UV^\tp $. Since the rank of $V^{(i)}\Lambda$ is $s<\min\{r,n_i-1\}$, the polar decomposition of the matrix $V^{(i)}\Lambda$ is not unique and has the form
\[
V^{(i)}\Lambda=P(V\Sigma V^\tp ),
\]
where
\[
P=U_1V_1^\tp +U_2QV_2^\tp \in V(r,n_i) \ \text{for some}\ Q\in \O(r-s).
\]

Since $\min\{r,n_i-1\}>s$, {we must have $n_i-s\geq 2$ and this implies that $U_2$} can be chosen from the following set
\[
C:=\{W_2\in  {\mathbb R^{n_i\times (r-s)} }\colon [U_1\ W_2 ]\in V(r,n_i)\}.
\]
{
Since $U_1$ is a fixed element in $V(s,n_i)$, $C$ is isomorphic to $V(r-s,n_i -s)$ and hence $C\subseteq \mathbb{R}^{n_i \times (r-s)}$ is an irreducible closed subvariety of dimension
\[
\dim C = \frac{1}{2}(r-s)\left( (n_i  - r)  + (n_i -s -1) \right)\ge 1.
\]
}
Therefore, in {any small} neighborhood of $P$, there exists an {orthonormal} matrix $\tilde P$ such that $\tilde P$ also gives a polar orthonormal factor matrix of $V^{(i)}\Lambda$. Now if we fix the other $U^{(i)}$'s and $\Upsilon$, then
\begin{align*}
g((U^{(1)},\dots,U^{(i-1)},\tilde P,U^{(i+1)},\dots,U^{(k)}),\operatorname{Diag}(\Lambda))&=\|\mathcal A\|^2-2\langle V^{(i)}\Lambda,\tilde P\rangle+\|\Upsilon\|^2\\
&=\|\mathcal A\|^2-2\langle V^{(i)}\Lambda,P\rangle+\|\Upsilon\|^2,
\end{align*}
where the second equality follows from the fact that $\tilde P$ is also a polar orthonormal factor matrix of $V^{(i)}\Lambda$. {Therefore, $g$ is constant at such a point
\[
((U^{(1)},\dots,U^{(i-1)},\tilde P,U^{(i+1)},\dots,U^{(k)}),\operatorname{Diag}(\Upsilon)).
\]
Since $(\mathbb U,\Upsilon)$ is a local minimizer of \eqref{eq:sota} {(or equivalently, $\mathbb U$ is a local maximizer of \eqref{eq:sota-max})}, we may conclude that
{such a point} is also a local minimizer of \eqref{eq:sota}. In particular, each such point
corresponds to a critical point of $g$ on $U_{\mathbf n,r}$ by Proposition~\ref{prop:licq} and Proposition~\ref{prop:critical-equivalence}. However, {this contradicts to} the fact that $(\mathbb U,\operatorname{Diag}(\Upsilon))$ is an isolated critical point of $g$ on $U_{\mathbf n,r}$.}
\end{proof}

\begin{corollary}\label{cor:proximal}
For a generic tensor $\mathcal A$, {there exist $\tau > \epsilon>0$ such that if Substep 2 in Algorithm~\ref{algo} is replaced by Algorithm~\ref{algo:revised},} then the proximal step (i.e., Case (ii) in Algorithm~\ref{algo:revised}) will only be executed finitely many times if the {algorithm converges} to a local maximizer of \eqref{eq:sota-max}.
\end{corollary}

\begin{proof}
For a generic tensor, there are finitely many essential KKT points whose corresponding primitive KKT points are all nondegenerate by Theorem~\ref{thm:kkt}. Therefore, there exists a constant $\tau>0$ such that it is strictly smaller than the smallest positive singular values of all $V^{(i)}\Lambda$'s {determined} by these primitive KKT points.

{We take $0< \epsilon<\tau$ and let $\{\mathbb U_{[p]}\}$ be a sequence generated by the modified algorithm which converges to $\mathbb U^*$. Note that the convergence is guaranteed by Lemma~\ref{lem:relative-decrease} and results in Section~\ref{sec:APPDT}. Let $\Lambda^*$ be the limit of $\Lambda^{(i)}_{[p]}$ defined in \eqref{eq:lambda-orth} and $V^{(*,i)}$ the limit of $V^{(i)}_{[p]}$ defined in \eqref{eq:marixv}. The truncation iteration ensures that $\operatorname{Diag}(\Lambda^*)$ is a vector with each component nonzero.}
Thus, by Theorem~\ref{thm:regular}, the limit point $\mathbb U^*$ is a regular KKT point of \eqref{eq:sota-max}. Consequently, the rank of the matrix $V^{(*,i)}\Lambda^*$ is either of full rank and hence the proximal correction step will not be executed by the choice of $\epsilon$ and $\tau$, or the rank of the matrix $V^{(*,i)}\Lambda^*$ is of rank $(r-1)$ when $r=n_i$, in which case Case (i) in Algorithm~\ref{algo:revised} will be executed by the choice of $\epsilon$ and $\tau$. Therefore, {after finitely many iterations}, Case (ii) in Algorithm~\ref{algo:revised} will not be executed.
\end{proof}
\subsection{Truncation}\label{sec:truncation}
In this subsection, {we will prove that for almost all LROTA problems, local minimizers {of \eqref{eq:sota}} are actually contained in the manifold $D(\mathbf n,r)$}. Therefore, if Algorithm~\ref{algo} converges to a local minimizer of \eqref{eq:sota}, we can choose a suitable $\kappa>0$ such that the truncation step (i.e., Step 2) in Algorithm~\ref{algo} is unnecessary.

\begin{theorem}\label{thm:minimizer}
If the sequence $\mathbf{n} = (n_1,\dots, n_k)$ and the positive integer $r\le \min\{n_1,\dots, n_k\}$ satisfies the relation
\begin{equation}\label{eqn:thm:minimizer:condition}
d_{\mathbf{n},r-1} < \prod_{i=1}^k (n_i - r+1),
\end{equation}
where $d_{\mathbf{n},r-1} \coloneqq (r-1)\left[\sum_{i=1}^kn_i-\frac{kr}{2}+1\right]$, then for a generic $\mathcal{A}\in \mathbb{R}^{n_1 \times \cdots \times n_k}$, each local minimizer of problem \eqref{eq:sota} is of the form $(U^{(1)}, \dots, U^{(k)},(\lambda_1,\dots,\lambda_r)) \in \V(r,n_1)\times \cdots \times \V(r,n_k)\times\mathbb R^r$ such that
\[
(U^{(1)}, \dots, U^{(k)})  \cdot \operatorname{diag}(\lambda_1,\dots, \lambda_r) \in D(\mathbf{n},r).
\]
\end{theorem}
\begin{proof}
We consider the subset $Z \subseteq \mathbb{R}^{n_1 \times \cdots \times n_k}$ consisting of tensors of the form
\[
(V^{(1)},\dots, V^{(k)}) \cdot \operatorname{diag}(\mu_1,\dots, \mu_{r-1}) + \mathcal{X}.
\]
Here $V^{(i)} \in \V(r-1,n_i)$, $\mu_j \in \mathbb{R}$, and $\mathcal{X}$ is a linear combination of decomposable tensors $\mathbf v_1 \otimes \cdots \otimes \mathbf v_k$ where for each $i=1,\dots, k$,
\begin{enumerate}
\item $\mathbf v_i\in \mathbb{R}^{n_i}$ is a unit norm vector;
\item if $\mathbf v_i$ is not a column vector of $V^{(i)}$, then $\mathbf v_i^\tp V^{(i)} = \mathbf 0$;
\item and there exists some $1\le j \le k$ such that $\mathbf v_{j}$ is a column vector of $V^{(j)}$.
\end{enumerate}

We notice that
\[
(V^{(1)},\dots, V^{(k)}) \cdot \operatorname{diag}(\mu_1,\dots, \mu_{r-1}) \in C(\mathbf{n},r-1)
\]
and $\mathcal{X}$ is contained in a vector space of dimension at most
\[
\prod_{i=1}^k n_i - \prod_{i=1}^k (n_i - r+1).
\]
This implies that the dimension of $Z$ is bounded above by
\[
\dim C(\mathbf{n},r-1) + \left( \prod_{i=1}^k n_i - \prod_{i=1}^k (n_i - r + 1) \right) = d_{\mathbf{n},r-1} +  \left( \prod_{i=1}^k n_i - \prod_{i=1}^k (n_i - r + 1) \right),
\]
since $\dim C(\mathbf{n},r-1) = d_{\mathbf{n},r-1}$ by Proposition~\ref{prop:smooth}.
In particular, we have
\[
\dim \overline{Z} = \dim Z < \prod_{i=1} n_i,
\]
where $\overline{Z}$ is the Zariski closure of $Z$.
Next we suppose that $\mathcal{A} \in U \coloneqq \mathbb{R}^{n_1 \times \cdots \times n_k} \setminus \overline{Z}$ and there exist $(V^{(1)},\dots,V^{(k)}) \in \V(r-1,n_1) \times \cdots \times \V(r-1,n_k)$ and $(\mu_1,\dots, \mu_{r-1})\in \mathbb{R}^{r-1}$ such that $(V^{(1)},\dots,V^{(k)},(\mu_1,\dots, \mu_{r-1}))$ is a local minimizer of \eqref{eq:sota}. We can write
\[
\mathcal{X} \coloneqq \mathcal{A} - (V^{(1)},\dots, V^{(k)})\cdot \operatorname{diag}(\mu_1,\dots, \mu_{r-1})
\]
as a linear combination of decomposable tensors $\mathbf v_1 \otimes \cdots \otimes \mathbf v_k$, such that
\begin{enumerate}
\item $\mathbf v_i\in \mathbb{R}^{n_i}$ is a unit norm vector;
\label{item:thm:minimizer:1}
\item either of the following occurs:
\begin{enumerate}
\item for each $i=1,\dots, k$, $\mathbf v_i$ is a column vector of $V^{(i)}$;
\label{item:thm:minimizer:2a}
\item for each $i=1,\dots, k$, $\mathbf v_i^\tp V^{(i)} = \mathbf 0$.
\label{item:thm:minimizer:2b}
\end{enumerate}
\end{enumerate}
According to the choice of $\mathcal{A}$, there exists $\mathbf v_1 \otimes \cdots \otimes \mathbf v_k$ satisfying \eqref{item:thm:minimizer:1} and \eqref{item:thm:minimizer:2b} such that
\[
\langle \mathcal{X}, \mathbf v_1 \otimes \cdots \otimes \mathbf v_k \rangle \ne 0.
\]
Now we set
\[
\mathcal{Y}\coloneqq (V^{(1)},\dots, V^{(k)})\cdot \operatorname{diag}(\mu_1,\dots, \mu_{r-1}) + \epsilon  \mathbf v_1 \otimes \cdots \mathbf v_k\in D(\mathbf{n},r),
\]
for a sufficiently small positive number $\epsilon$. We have
\[
\lVert \mathcal{A} - \mathcal{Y} \rVert = \lVert \mathcal{X} - \epsilon \mathbf v_1 \otimes \cdots \otimes \mathbf v_k \rVert < \lVert \mathcal{X} \rVert.
\]
This contradicts the assumption that $(V^{(1)},\dots,V^{(k)},(\mu_1,\dots, \mu_{r-1}))$ is a local minimizer of problem \eqref{eq:sota}.
\end{proof}

As a special case, we suppose that $n_1 = \cdots = n_k = n$ so \eqref{eqn:thm:minimizer:condition} is written as
\[
\left(  r-1 \right) \left( k (n - \frac{r}{2}) + 1 \right) < (n- r + 1)^k.
\]
We set $r - 1 = (1 -\alpha) n$ for $\alpha \in [\frac{1}{n},1]$, hence we have
\[
(1 -\alpha) n \left( \frac{kn(1+\alpha)}{2} + 1 - \frac{k}{2} \right) <  \alpha^k n^k.
\]
Therefore, to guarantee \eqref{eqn:thm:minimizer:condition} in this case, it is sufficient to require
\begin{equation}\label{eqn:sufficient condition}
  \frac{2   n^{k-2}}{k} \alpha^k + \alpha^2 - 1 > 0.
\end{equation}

\begin{corollary}
If $n_1 = \cdots = n_k = n$ and {
\[
1 \le r  \le \left( 1 - \left( \frac{k}{2   n^{k-2}} \right)^{\frac{1}{k}} \right)n + 1,
\]}
then for a generic $\mathcal{A}\in \mathbb{R}^{n \times \cdots \times n}$, any local minimizer of the problem \eqref{eq:sota} lies in $D(\mathbf n,r)$ (or equivalently $U_{\mathbf n,r}$). In particular, for any fixed $k$ {and $r$, there exists $n_0$ such that whenever $n \ge n_0$ and $\mathcal{A}\in \mathbb{R}^{n \times \cdots \times n}$ is generic, any local minimizer of problem \eqref{eq:sota} lies in $D(\mathbf n,r)$.}
\end{corollary}
\begin{proof}
We observe that for any $(k/2n^{k-2})^{1/k} \le  \alpha \le 1  $, \eqref{eqn:sufficient condition} and hence \eqref{eqn:thm:minimizer:condition} is satisfied by $n_1 = \cdots = n_k = n$ and
\[
r = (1-\alpha)n + 1.
\]
This implies that for
\[
1 \le r  \le \left( 1 - \left( \frac{k}{2   n^{k-2}} \right)^{\frac{1}{k}} \right)n + 1,
\]
any local minimizer of the problem \eqref{eq:sota} lies in $D(\mathbf n,r)$ for a generic $\mathcal{A}$. In particular, if $k\ge 3$ is a fixed integer, then
{
\[
\lim_{n\to \infty}  \frac{ \left( 1 - \left( \frac{k}{2   n^{k-2}} \right)^{\frac{1}{k}} \right)n + 1}{n} = 1.
\]
This implies that there exists some integer $n_0$ such that for a generic $\mathcal{A}$ and any $n\ge n_0$, local minimizers of problem \eqref{eq:sota} must all lie in $D(\mathbf n,r)$.
}
\end{proof}

If \eqref{eqn:thm:minimizer:condition} is fulfilled,
Proposition~\ref{thm:minimizer} implies that all local maximizers of \eqref{eq:sota-max} for a generic tensor are in $D(\mathbf n,r)$. Recall from Theorem~\ref{thm:kkt} that these local maximizers are finite. Thus, we have the next corollary.
\begin{corollary}\label{cor:truncation}
Let {$n_1,\dots,n_k$} and $r$ be positive integers satisfying\eqref{eqn:thm:minimizer:condition}.
For a generic tensor $\mathcal A$, if Algorithm~\ref{algo} converges to a local maximizer, then the truncation step in Algorithm~\ref{algo} will not be executed when a suitable $\kappa$ is chosen.
\end{corollary}

Combining Corollaries~\ref{cor:proximal} and \ref{cor:truncation}, we have the following conclusion.
\begin{proposition}\label{prop:iapd-apd}
For almost all LROTA problems, {there exist $\kappa$, $\epsilon$ and $\tau$ such that iAPD with Substep 2 being replaced by Algorithm~\ref{algo:revised} reduces to APD after finitely many iterations.}
\end{proposition}

\section{Conclusions}\label{sec:final}
{In this paper, we propose} an alternating polar decomposition algorithm with adaptive proximal correction and truncation for approximating a given tensor by a low rank orthogonally decomposable tensor. {Without any assumption we prove that this algorithm has global convergence and overall sublinear convergence with a sub-optimal explicit convergence rate. For a generic tensor, this algorithm converges $R$-linearly without any further assumption. For the first time, the convergence rate analysis for the problem of low {rank} orthogonal tensor approximations is accomplished.}

{The discussion in Section~\ref{sec:discussion} is all about local maximizers of problem \eqref{eq:sota-max}.} Both {APD and iAPD are based} on the alternating minimization method \cite{B-99}, which is a variant of gradient ascend. In general, {such a method can only converge to a KKT point}, including local minimizer, saddle point and local maximizer \cite{B-99,Beck-book}. However, if each saddle point of a function has the \textit{strict saddle property}, which is guaranteed if it is a nondegenerate critical point, then with probability one, the gradient ascent method converges to a local maximizer \cite{LSJR-16}. Therefore, for a generic tensor, the proposed algorithm will return a strict local maximizer, since each KKT point is nondegenerate and the objective function is monotonically increasing. More theoretical investigations on this {are} interesting and important, {which will be our next project}.
\subsection*{Acknowledgement}
This work is partially supported by National Science Foundation of China (Grant No. 11771328). The first author is also partially supported Young Elite Scientists Sponsorship Program by Tianjin and the Natural Science Foundation of Zhejiang Province, China (Grant No. LD19A010002). The second author is also partially supported by National Science Foundation of China (Grant No.~11801548 and Grant No.~11688101), National Key R\&D Program of China (Grant No.~2018YFA0306702) and the Recruitment Program of Global Experts of China.


\appendix
\renewcommand{\appendixname}{Appendix~\Alph{section}}
\section{Properties on Orthonormal Matrices}\label{sec:polar}
\subsection{Polar decomposition}
In this section, we will {establish} an error bound analysis for the polar decomposition. For a positive semidefinite matrix {$H\in\operatorname{S}^{n}_+$, there exists a unique positive semidefinite matrix $P\in\operatorname{S}^{n}_+$} such that $P^2=H$. In the literature, this matrix $P$ is called the \textit{square root} of the matrix $H$ and denoted as $P=\sqrt{H}$. If $H=U\Sigma U^\tp $ is the eigenvalue decomposition of $H$, then we have $\sqrt{H}=U\sqrt{\Sigma}U^\tp $, where $\sqrt{\Sigma}$ is the diagonal matrix {whose diagonal elements are square roots of those of $\Sigma$.} The next result is classical, which can be found in \cite{H-86,GV-13}.
\begin{lemma}[Polar Decomposition]\label{lem:polar}
Let $A\in\mathbb R^{m\times n}$ with $m\geq n$. Then there exist an orthonormal matrix $U\in V(n,m)$ and a unique symmetric positive semidefinite matrix $H\in \operatorname{S}^{n}_+$ such that $A=UH$ and
\begin{equation}\label{eq:polar-optimal}
U\in\operatorname{argmax}\{\langle Q,A\rangle\colon Q\in V(n,m)\}.
\end{equation}
Moreover, if $A$ is of full rank, then the matrix $U$ is uniquely determined and $H$ is positive definite.
\end{lemma}

The matrix decomposition $A=UH$ as in Lemma~\ref{lem:polar} is called the \textit{polar decomposition} of the matrix $A$ \cite{GV-13}. For convenience, the matrix $U$ is referred as {a} \textit{polar orthonormal factor matrix} and the matrix $H$ is the \textit{polar positive semidefinite factor matrix}.
The optimization reformulation \eqref{eq:polar-optimal} comes from the approximation problem
\[
\min_{Q\in V(n,m)}\ \|B-QC\|^2
\]
for two given matrices $B$ and $C$ of {proper sizes.} In the following, we give a global error bound for this problem. To this end, the next lemma is useful.
\begin{lemma}[Error Reformulation]\label{lem:polar-error-full}
{Let $p, m, n$ be positive integers with $m\ge n$ and let $B\in\mathbb R^{m\times p}$, $C\in\mathbb R^{n\times p}$ be two given matrices. We set $A \coloneqq BC^\tp \in\mathbb R^{m\times n}$ and suppose that $A=WH$ is a polar decomposition of $A$. We have}
\begin{equation}\label{eq:polar-error-full}
\|B-QC\|_F^2-\|B-WC\|_F^2= \|W\sqrt{H}-Q\sqrt{H}\|^2_F
\end{equation}
for any orthonormal matrix $Q\in V(n,m)$.
\end{lemma}

\begin{proof}
We have
\begin{align*}
\|B-QC\|_F^2-\|B-WC\|_F^2&=2\langle B,WC-QC\rangle\\
&=2\langle A,W-Q\rangle\\
&=2\langle WH,W-Q\rangle\\
&= \langle WH,W\rangle-2\langle WH,Q\rangle+\langle QH,Q\rangle\\
&= \|W\sqrt{H}-Q\sqrt{H}\|^2_F,
\end{align*}
where both the first and the fourth equalities follow from the fact that both $Q$ and $W$ are in $V(n,m)$, and the {last one is derived} from the fact that $H$ is symmetric and positive semidefinite by Lemma~\ref{lem:polar}.
\end{proof}

Given an $n\times n$ symmetric positive semidefintie matrix $H$, we can define a {symmetric bilinear form on $\mathbb R^{m\times n}$ by}
\begin{equation}\label{eq:inner}
\langle P,Q\rangle_H:=\langle PH,Q\rangle
\end{equation}
for all $m\times n$ matrices $P$ and $Q$. It can also induce {a seminorm}
\begin{equation}\label{eq:norm}
\|A\|_H:=\sqrt{\langle A,A\rangle_H}=\|A\sqrt{H}\|_F.
\end{equation}
{In particular, if $H$ is positive definite, then $\lVert \cdot \rVert_H$ is a norm on $\mathbb{R}^{m\times n}$.} Thus, the error estimation in \eqref{eq:polar-error-full} can be viewed as a distance estimation between $W$ and $Q$ {with respect to the distance induced by this norm. Moreover, if $H$ is the identity matrix, then $\lVert \cdot \rVert_H$ is simply the Frobenius norm which induces the Euclidean distance on $\mathbb{R}^{m\times n}$.} By Lemma~\ref{lem:polar-error-full} it is easy to see that the optimizer in \eqref{eq:polar-optimal} is unique whenever $A$ is of full rank.

The following result {establishes} the error estimation {with respect to the Euclidean distance.} Given a matrix $A\in\mathbb R^{m\times n}$, let $\sigma_{\min}(A)$ be the smallest singular value of $A$. If $A$ is of full rank, then $\sigma_{\min}(A)>0$.

\begin{theorem}[Global Error Bound in Frobenius Norm]\label{thm:error}
{Let $p, m, n$ be positive integers with $m\ge n$ and let $B\in\mathbb R^{m\times p}$ and $C\in\mathbb R^{n\times p}$ be two given matrices. We set $A:=BC^\tp \in\mathbb R^{m\times n}$ and suppose that $A$ is of full rank with the polar decomposition $A=WH$. We have that}
\begin{equation}\label{eq:polar-error}
\|B-QC\|_F^2-\|B-WC\|_F^2\geq \sigma_{\min}(A)\|W-Q\|^2_F
\end{equation}
for any orthonormal matrix $Q\in V(n,m)$.
\end{theorem}

\begin{proof}
We know that in this case
\[
\sqrt{H}-\sqrt{\sigma_{\min}(A)}I\in\operatorname{S}^{n}_+.
\]
Therefore we may conclude that
\begin{align*}
\|W\sqrt{H}-Q\sqrt{H}\|^2_F&=\|W-Q\|^2_{\sqrt{H}}\\
&\geq \|W-Q\|^2_{\sqrt{\sigma_{\min}(A)}I}\\
&=\|W\sqrt{\sigma_{\min}(A)}I-Q\sqrt{\sigma_{\min}(A)}I\|^2_F\\
&=\sigma_{\min}(A)\|W-Q\|_F^2.
\end{align*}
According to Lemma~\ref{lem:polar-error-full}, we {can} derive the desired inequality.
\end{proof}

Theorem~\ref{thm:error} is a refinement of Sun and Chen's result (cf.\ \cite[Theorem~4.1]{SC-89}), {in} which the right hand side of \eqref{eq:polar-error} has an extra factor $\frac{1}{4}$.

\subsection{Principal angles between subspaces}\label{sec:principal}
Given two linear subspaces $\mathbb U$ and $\mathbb V$ of dimension $r$ in $\mathbb R^n$, the \textit{principal angles} $\{\theta_i\colon i=1,\dots,r\}$ between {$\mathbb{U}$ and $\mathbb{V}$} and the associated \textit{principal vectors} $\{(\mathbf u_i,\mathbf v_i)\colon i=1,\dots,r\}$ are defined recursively by
\begin{equation}\label{eq:pricp-ag}
\cos(\theta_i)=\langle\mathbf u_i,\mathbf v_i\rangle= \max_{\mathbf u\in \mathbb U,\ [\mathbf u,\mathbf u_1,\dots,\mathbf u_{i-1}]\in V(i,n)}{\left\{\max_{\mathbf v\in \mathbb V,\ [\mathbf v,\mathbf v_1,\dots,\mathbf v_{i-1}]\in V(i,n)} \langle\mathbf u,\mathbf v\rangle\right\}}.
\end{equation}

The following result is {standard, whose proof} can be found in \cite[Section~6.4.3]{GV-13}.
\begin{lemma}\label{lem:theta}
{For any orthonormal matrices $U,V\in V(r,n)$ of which subspaces spanned by column vectors are $\mathbb{U}$ and $\mathbb{V}$ respectively, we have}
\begin{equation}\label{eq:singular-ang}
\sigma_i(U^\tp V)=\cos(\theta_i)\ \text{for all }i=1,\dots,r,
\end{equation}
where $\theta_i$'s are the principal angles between $\mathbb U$ and $\mathbb V$, {and $\sigma_i(U^\tp V)$ is the $i$-th largest singular value of the matrix $U^\tp V$.}
\end{lemma}

\begin{lemma}\label{lem:inequality}
{For any orthonormal matrices $U,V\in V(r,n)$, we have}
\begin{equation}\label{eq:trace}
\langle U, V \rangle \le  \sum_{j=1}^r\sigma_j(U^{\tp }V).
\end{equation}
\end{lemma}

\begin{proof}
We recall from \cite[pp.~331]{GV-13} that
\[
\min_{Q\in \O(r)} \| A - BQ \|^2_F = \sum_{j=1}^r (\sigma_j(A)^2 - 2 \sigma_j(B^{\tp }A) + \sigma_j(B)^2),
\]
for any $n\times r$ matrices $A,B$. In particular, if $U,V\in \V(r,n)$ then
\[
\sigma_j(U) = \sigma_j(V) = 1\ \text{for all }j=1,\dots,r, \quad \| U \|^2_F = \| V \|^2_F = r.
\]
{
This implies that
\[
2 r - 2 \sum_{j=1}^r \sigma_j(U^\tp V) = \min_{Q\in \O(r)} \| U - VQ \|^2_F \le \| U - V \|^2_F = 2r - 2 \langle U, V \rangle,
\]
and the desired inequality follows immediately.
}
\end{proof}

\begin{lemma}\label{lem:distance}
For any orthonormal matrices $U,V\in V(r,n)$, we have
\begin{equation}\label{eq:distance}
\|U^\tp V-I\|_F^2 \leq \|U-V\|^2_F.
\end{equation}
\end{lemma}

\begin{proof}
We have
\[
\|U^\tp V-I\|_F^2 = r + \sum_{i=1}^r \cos^2 \theta_i  - 2 \tr(U^\tp V) \leq 2r-2\tr(U^\tp V)=\|U-V\|_F^2,
\]
where the first equality follows from Lemma~\ref{lem:theta}.
\end{proof}

Let $\mathbb U^\perp$ be the orthogonal complement subspace of a given linear subspace $\mathbb U$ in $\mathbb R^n$.
A useful fact about principal angles between two linear subspaces $\mathbb U$ and $\mathbb V$ and those between $\mathbb U^\perp$ and $\mathbb V^\perp$ is stated as follows.
The proof can be found in \cite[Theorem~2.7]{KA-07}.
\begin{lemma}\label{lem:angles-orth}
Let $\mathbb U$ and $\mathbb V$ be two linear subspaces of the same dimension and let $\frac{\pi}{2}\geq \theta_s \geq\dots\geq\theta_1 > 0$ be {the nonzero principal angles} between $\mathbb U$ and $\mathbb V$. Then {the nonzero principal angles} between $\mathbb U^\perp$ and $\mathbb V^\perp$ are $\frac{\pi}{2}\geq \theta_s \geq\dots\geq\theta_1 >  0$.
\end{lemma}

The following result is for the general case, which might be of {independent interests}.
\begin{lemma}\label{lem:error-def}
Let $m\geq n$ be positive integers and let $V:=[V_1\ V_2]\in\O(m)$ with $V_1\in V(n,m)$ and $U\in V(n,m)$ be two given orthonormal matrices. Then, there exists an orthonormal matrix $W\in V(m-n,m)$ such that $P:=[U\ W]\in \O(m)$ and
\begin{equation}\label{eq:kapp}
\|P-V\|_F^2\leq 2\|U-V_1\|_F^2.
\end{equation}
\end{lemma}

\begin{proof}
By {a simple computation, it is straightforward to verify that} \eqref{eq:kapp} is equivalent to
\begin{equation}\label{eq:kapp-ob}
\|W-V_2\|_F^2\leq \| U-V_1\|_F^2.
\end{equation}
{To that end, we let} $U_2\in V(m-n,m)$ be an orthonormal matrix such that $[U\  U_2]\in\O(m)$. Then, we have that the linear subspace $\mathbb U_2$ {spanned by column vectors of $U_2$ is the orthogonal complement of $\mathbb U_1$}, {which is spanned by} column vectors of $U$. Likewise, let $\mathbb V_1$ and $\mathbb V_2$ be linear subspaces spanned by column vectors of $V_1$ and $V_2$ respectively.

Let $\frac{\pi}{2}\geq\theta_s\geq\dots\geq\theta_1 > 0$ be {the nonzero principal angles} between $\mathbb U_2$ and $\mathbb V_2$ for some nonnegative integer $s\leq m-n$. We have by Lemmas~\ref{lem:theta}, and \ref{lem:inequality} that
\begin{equation}\label{eq:error-1}
\langle U_2,V_2\rangle\leq\sum_{i=1}^{m-n}\sigma_i(U_2^\tp V_2)=\sum_{i=1}^s\cos(\theta_i)+(m-n)-s.
\end{equation}
Let $Q\in \O(m-n)$ be a polar orthogonal factor matrix of the matrix $U_2^\tp V_2$. It follows from the property of polar decomposition that
\begin{equation}\label{eq:error-2}
\langle U_2Q,V_2\rangle = \sum_{i=1}^{m-n}\sigma_i(U_2^\tp V_2).
\end{equation}

On the other hand, {nonzero principal angles} between $\mathbb U_1$ and $\mathbb V_1$ are $\frac{\pi}{2}\geq\theta_s\geq\dots\geq\theta_1>0$ by Lemma~\ref{lem:angles-orth}. Therefore, by Lemmas~\ref{lem:theta}, and \ref{lem:inequality}, we have that
\begin{equation}\label{eq:error-3}
\langle U,V_1\rangle\leq \sum_{i=1}^{n}\sigma_i(U^\tp V_1)=\sum_{i=1}^s\cos(\theta_i)+n-s.
\end{equation}

In a conclusion, if we set $W \coloneqq U_2Q$, then we have the following:
\begin{align*}
\|W-V_2\|_F^2&=2(m-n)-2\langle U_2Q,V_2\rangle\\
&=2(m-n)-2\big(\sum_{i=1}^s\cos(\theta_i)+m-n-s\big)\\
&=2n - 2\big(\sum_{i=1}^s\cos(\theta_i)+n-s\big)\\
&\leq 2n -2\langle U,V_1\rangle\\
&=\|U-V_1\|^2_F,
\end{align*}
where the second equality follows from \eqref{eq:error-1} and \eqref{eq:error-2} and the inequality follows from \eqref{eq:error-3}.
\end{proof}
%
\section{Proofs of Technical Lemmas in Section~\ref{sec:APPDT}}\label{app:proofalgo}
\subsection{Proof of Lemma~\ref{lem:lambda-k}}\label{app:proofalgo-lambadk}

\begin{proof}
For each $i\in\{0,\dots,k-1\}$, we have
\begin{align}
f(\mathbb U_{i+1,[p]})-f(\mathbb U_{i,[p]})&=\sum_{j=1}^r(\lambda^{i+1}_{j,[p]})^2-\sum_{j=1}^r(\lambda^{i}_{j,[p]})^2\nonumber\\
&=\sum_{j=1}^r(\lambda^{i+1}_{j,[p]}+\lambda^{i}_{j,[p]})(\lambda^{i+1}_{j,[p]}-\lambda^{i}_{j,[p]})\nonumber\\
&=\sum_{j=1}^r\lambda^{i+1}_{j,[p]}(\lambda^{i+1}_{j,[p]}-\lambda^{i}_{j,[p]})+\sum_{j=1}^r\lambda^{i}_{j,[p]}(\lambda^{i+1}_{j,[p]}-\lambda^{i}_{j,[p]}).\label{eq:obj-k-exp}
\end{align}

{We first analyze the second summand in \eqref{eq:obj-k-exp} and by considering the following two cases:}
\begin{enumerate}
\item\label{item:proof of 4.2} If $\sigma^{(i+1)}_{r,[p]}\geq \epsilon$, then there is no proximal step in Algorithm~\ref{algo} and we have that $\sigma_{\min}(S^{(i+1)}_{[p]})=\sigma^{(i+1)}_{r,[p]}\geq \epsilon$,
where $S^{(i+1)}_{[p]}$ is the polar positive semidefinite factor matrix of $V^{(i+1)}_{[p]}\Lambda^{(i+1)}_{[p]}$ {obtained in \eqref{eq:polar}.} {From \eqref{eq:inter-vector}, \eqref{eq:lambda-orth} and \eqref{eq:marixv} we notice that
\begin{align*}
((U^{(i+1)}_{[p]})^\tp V^{(i+1)}_{[p]}\Lambda^{(i+1)}_{[p]})_{jj} &=((\mathbf u^{(i+1)}_{j,[p]})^\tp\mathbf v^{(i+1)}_{j,[p]})\lambda^{i}_{j,[p]} \\
&=((\mathbf u^{(i+1)}_{j,[p]})^\tp\mathcal A\tau_{i+1}\mathbf x^{(i+1)}_{j,[p]})\lambda^{i}_{j,[p]} \\
& =\mathcal A\tau(\mathbf x^{(i+2)}_{j,[p]})\lambda^{i}_{j,[p]} \\
& =\lambda^{i+1}_{j,[p]}\lambda^{i}_{j,[p]},
\end{align*}
and similarly $((U^{(i+1)}_{[p-1]})^\tp V^{(i+1)}_{[p]}\Lambda^{(i+1)}_{[p]})_{jj} = \lambda^{i}_{j,[p]}\lambda^{i}_{j,[p]}$.} {Hence by Lemma~\ref{lem:polar-error-full}, we obtain}
\begin{align}
\sum_{j=1}^r\lambda^{i}_{j,[p]}(\lambda^{i+1}_{j,[p]}-\lambda^{i}_{j,[p]})&=\operatorname{Tr}((U^{(i+1)}_{[p]})^\tp V^{(i+1)}_{[p]}\Lambda^{(i+1)}_{[p]})
-\operatorname{Tr}((U^{(i+1)}_{[p-1]})^\tp V^{(i+1)}_{[p]}\Lambda^{(i+1)}_{[p]})\nonumber \\
&=\frac{1}{2}\big\|(U^{(i+1)}_{[p]}-U^{(i+1)}_{[p-1]})\sqrt{S^{(i+1)}_{[p]}}\big\|_F^2\nonumber\\
&\geq \frac{\epsilon}{2}\|U^{(i+1)}_{[p]}-U^{(i+1)}_{[p-1]}\|_F^2\label{eq:bound-ineq}\\
&\geq 0.\nonumber
\end{align}

\item If $\sigma^{(i+1)}_{r,[p]}<\epsilon$, we consider the following matrix optimization problem
\begin{equation}\label{eq:proximal}
\begin{array}{rl}
\max&\langle V^{(i+1)}_{[p]}\Lambda^{(i+1)}_{[p]},U\rangle-\frac{\epsilon}{2}\|U-U^{(i+1)}_{[p-1]}\|_F^2\\
\text{s.t.}& U\in V(r,n_{i+1}).
\end{array}
\end{equation}
{Since $U,U^{(i+1)}_{[p-1]} \in {V(r,n_{i+1})}$, we must have}
\[
\frac{\epsilon}{2}\|U- U^{(i+1)}_{[p-1]}\|_F^2=\epsilon r-\epsilon\langle U^{(i+1)}_{[p-1]},U\rangle.
\]
Thus, by Lemma~\ref{lem:polar}, a global maximizer of \eqref{eq:proximal} is given by a polar orthonormal factor matrix of the matrix $V^{(i+1)}_{[p]}\Lambda^{(i+1)}_{[p]}+\epsilon U^{(i+1)}_{[p-1]}$. By Substep 2 of Algorithm~\ref{algo}, $U^{(i+1)}_{[p]}$ is a polar orthonormal factor matrix of the matrix $V^{(i+1)}_{[p]}\Lambda^{(i+1)}_{[p]}+\epsilon U^{(i+1)}_{[p-1]}$, and hence a global maximizer of \eqref{eq:proximal}. Thus,
by the optimality of $U^{(i+1)}_{[p]}$ for \eqref{eq:proximal}, we have
\[
\langle V^{(i+1)}_{[p]}\Lambda^{(i+1)}_{[p]},U^{(i+1)}_{[p]}\rangle-\frac{\epsilon}{2}\|U^{(i+1)}_{[p]}-U^{(i+1)}_{[p-1]}\|_F^2\geq \langle V^{(i+1)}_{[p]}\Lambda^{(i+1)}_{[p]},U^{(i+1)}_{[p-1]}\rangle.
\]
{
Therefore, the inequality \eqref{eq:bound-ineq} in case \eqref{item:proof of 4.2} also holds in this case.
}
\end{enumerate}

Consequently, we have
\begin{equation}\label{eq:byproduct}
0\leq \sum_{j=1}^r\lambda^{i}_{j,[p]}(\lambda^{i+1}_{j,[p]}-\lambda^{i}_{j,[p]})=\sum_{j=1}^r\lambda^{i+1}_{j,[p]}\lambda^{i}_{j,[p]}-\sum_{j=1}^r(\lambda^{i}_{j,[p]})^2,
\end{equation}
{which together with Cauchy-Schwartz inequality implies that}
\begin{equation}\label{eq:caychy}
\big(\sum_{j=1}^r(\lambda^{i}_{j,[p]})^2\big)^2\leq \big(\sum_{j=1}^r\lambda^{i}_{j,[p]}\lambda^{i+1}_{j,[p]} \big)^2\leq\sum_{j=1}^r(\lambda^{i}_{j,[p]})^2\sum_{j=1}^r(\lambda^{i+1}_{j,[p]})^2.
\end{equation}
Since $f(\mathbb U_{[0]})>0$, we conclude that
$f(\mathbb U_{0,[1]})=\sum_{j=1}^r(\lambda^{0}_{j,[1]})^2>0$ and hence $\sum_{j=1}^r\lambda^{1}_{j,[1]}\lambda^{0}_{j,[1]}>0$ by \eqref{eq:byproduct}. Thus, we conclude that
\begin{equation}\label{eqn:proof of 4.2-1}
\sum_{j=1}^r\lambda^{1}_{j,[1]}\lambda^{0}_{j,[1]} \leq \sum_{j=1}^r(\lambda^{1}_{j,[1]})^2\frac{\sum_{j=1}^r(\lambda^{0}_{j,[1]})^2}{\sum_{j=1}^r\lambda^{1}_{j,[1]}\lambda^{0}_{j,[1]}}\leq \sum_{j=1}^r(\lambda^{1}_{j,[1]})^2,
\end{equation}
where the first inequality follows from \eqref{eq:caychy} and the second from \eqref{eq:byproduct}. {Combining \eqref{eqn:proof of 4.2-1} with \eqref{eq:obj-k-exp} and \eqref{eq:bound-ineq}, we may obtain \eqref{eq:obj-k} for $i=0$ {and $p=1$}.}

{On the one hand, from \eqref{eq:caychy} we obtain
\[
0 \le \sum_{j=1}^r (\lambda_{j,[p]}^i)^2 (\sum_{j=1}^r (\lambda_{j,[p]}^{i+1})^2 - \sum_{j=1}^r (\lambda_{j,[p]}^i)^2).
\]
Since $\sum_{j=1}^r (\lambda_{j,[p]}^i)^2 > 0$ if there is no truncation, } {we must have
\[
0 \le \sum_{j=1}^r (\lambda_{j,[p]}^{i+1})^2 - \sum_{j=1}^r (\lambda_{j,[p]}^i)^2 = f(\mathbb U_{i+1,[p]})-f(\mathbb U_{i,[p]}),
\]
i.e., the objective function $f$ is monotonically increasing} during the APD iteration as long as there is no truncation. {On the other hand, there are at most $r$ truncations and hence the total loss of $f$ by the truncation} is at most $r\kappa^2<f(\mathbb U_{[0]})$.
{Therefore, $f$ is always positive and  according to \eqref{eq:byproduct}, we may conclude that $\sum_{j=1}^r\lambda^{i+1}_{j,[p]}\lambda^{i}_{j,[p]}>0$ along iterations.} By induction on $p$, we obtain
\[
\sum_{j=1}^r\lambda^{i+1}_{j,[p]}\lambda^{i}_{j,[p]} \leq \sum_{j=1}^r(\lambda^{i+1}_{j,[p]})^2\frac{\sum_{j=1}^r(\lambda^{i}_{j,[p]})^2}{\sum_{j=1}^r\lambda^{i+1}_{j,[p]}\lambda^{i}_{j,[p]}}\leq \sum_{j=1}^r(\lambda^{i+1}_{j,[p]})^2,
\]
which together with \eqref{eq:obj-k-exp} and \eqref{eq:bound-ineq}, implies \eqref{eq:obj-k} for arbitrary nonnegative integer $p$.
\end{proof}

\subsection{Proof for Lemma~\ref{lem:subdiff}}\label{app:proofalgo-sub}

\begin{proof}
The subdifferentials of $h$ can be partitioned {as follows:}
\begin{equation}\label{eq:sub-block}
\partial h(\mathbb U)=(\nabla_1 f(\mathbb U)+\partial\delta_{V(r,n_1)}(U^{(1)}))\times\dots\times (\nabla_k f(\mathbb U)+\partial\delta_{V(r,n_k)}(U^{(k)})).
\end{equation}
Following the notation of Algorithm~\ref{algo}, {we set}
\[
\mathbf x_{j}:=(\mathbf u^{(1)}_{j,[p+1]},\dots,\mathbf u^{(k)}_{j,[p+1]})\ \text{for all }j=1,\dots,r,
\]
where $\mathbf u^{(i)}_{j,[p+1]}$ is the {$j$-th} column of the matrix $U^{(i)}_{[p+1]}$ for all $i\in\{1,\dots,k\}$,
\begin{equation}\label{eq:matrix-vi}
V^{(i)}:=\begin{bmatrix}\mathbf v^{(i)}_1&\dots&\mathbf v^{(i)}_r\end{bmatrix}\ \text{with }\mathbf v^{(i)}_j:=\mathcal A\tau_i(\mathbf x_{j})\ \text{for all }j=1,\dots,r,
\end{equation}
for all $i\in\{1,\dots,k\}$ and
\begin{equation}\label{eq:matrix-lambda}
\Lambda:=\operatorname{diag}(\lambda_1,\dots,\lambda_r)\ \text{with }\lambda_j:=\mathcal A\tau(\mathbf x_{j}).
\end{equation}

{{By \eqref{eq:polar} and \eqref{eq:proximal}}, we have}
\begin{equation}\label{eqn:proof of lemma 4.5-1}
V^{(i)}_{[p+1]}\Lambda^{(i)}_{[p+1]}+\alpha U^{(i)}_{[p]}=U^{(i)}_{[p+1]}S^{(i)}_{[p+1]},
\end{equation}
where $\alpha\in\{0,\epsilon\}$ depending on {whether or not} there is a proximal correction. {According to \eqref{eq:normal-sub} and \eqref{eqn:proof of lemma 4.5-1}, we have
\[
-U^{(i)}_{[p+1]} \in \partial\delta_{V(r,n_i)}(U^{(i)}_{[p+1]}), \quad
V^{(i)}_{[p+1]}\Lambda^{(i)}_{[p+1]}+\alpha\big(U^{(i)}_{[p]}\big) \in \partial\delta_{V(r,n_i)}(U^{(i)}_{[p+1]}),
\]
which implies that $V^{(i)}_{[p+1]}\Lambda^{(i)}_{[p+1]}+\alpha\big(U^{(i)}_{[p]}-U^{(i)}_{[p+1]}\big)\in \partial\delta_{V(r,n_i)}(U^{(i)}_{[p+1]})$. If we take
\begin{equation}\label{eq:w}
W^{(i)}_{[p+1]}:=2V^{(i)}\Lambda -2V^{(i)}_{[p+1]}\Lambda^{(i)}_{[p+1]}-2\alpha\big(U^{(i)}_{[p]}-U^{(i)}_{[p+1]}\big),
\end{equation}
then we have}
\[
W^{(i)}_{[p+1]}\in 2V^{(i)}\Lambda +\partial\delta_{V(r,n_i)}(U^{(i)}_{[p+1]}).
\]

On the other hand,
\begin{align*}
&\quad\ \frac{1}{2}\|W^{(i)}_{[p+1]}\|_F\\
&=\|V^{(i)}\Lambda -V^{(i)}_{[p+1]}\Lambda^{(i)}_{[p+1]}-\alpha\big(U^{(i)}_{[p]}-U^{(i)}_{[p+1]}\big)\|_F\\
&\leq\|V^{(i)}\Lambda -V^{(i)}_{[p+1]}\Lambda \|_F+\|V^{(i)}_{[p+1]}\Lambda -V^{(i)}_{[p+1]}\Lambda^{(i)}_{[p+1]}\|_F+\alpha\|U^{(i)}_{[p]}-U^{(i)}_{[p+1]}\|_F\\
&\leq\|V^{(i)}-V^{(i)}_{[p+1]}\|_F\|\Lambda \|_F+\|V^{(i)}_{[p+1]}\|_F\|\Lambda -\Lambda^{(i)}_{[p+1]}\|_F+\alpha\|U^{(i)}_{[p]}-U^{(i)}_{[p+1]}\|_F\\
&\leq{\|\mathcal A\|}\|\Lambda \|_F\big(\sum_{j=1}^r\|\tau_i(\mathbf x_j)-\tau_i(\mathbf x^i_{j,[p+1]})\|\big)\\
&\quad +\|V^{(i)}_{[p+1]}\|_F\|\mathcal A\|\big(\sum_{j=1}^r\|\tau(\mathbf x_j)-\tau(\mathbf x^i_{j,[p+1]})\|\big)+\alpha\|U^{(i)}_{[p]}-U^{(i)}_{[p+1]}\|_F\\
&\leq \sqrt{r}\|\mathcal A\|^2\big(\sum_{j=1}^r\sum_{s=i+1}^{k}\|\mathbf u^{(s)}_{j,[p+1]}-\mathbf u^{(s)}_{j,[p]}\| \big)\\
&\quad+\sqrt{r}\|\mathcal A\|^2\big(\sum_{j=1}^r\sum_{s=i}^{k}\|\mathbf u^{(s)}_{j,[p+1]}-\mathbf u^{(s)}_{j,[p]}\|\big)+\alpha\|U^{(i)}_{[p]}-U^{(i)}_{[p+1]}\|_F\\
&\leq (2r\sqrt{r}\|\mathcal A\|^2 +\epsilon)\|\mathbb U_{[p+1]}-\mathbb U_{[p]}\|_F,
\end{align*}
where the third inequality follows from the fact that
\[
V^{(i)}-V^{(i)}_{[p+1]}=\begin{bmatrix}\mathcal A(\tau_i(\mathbf x_1)-\tau_i(\mathbf x^i_{1,[p+1]}))&\dots&\mathcal A(\tau_i(\mathbf x_r)-\tau_i(\mathbf x^i_{r,[p+1]}))\end{bmatrix},
\]
and a similar formula for $\Lambda -\Lambda^{(i)}_{[p+1]}$, the fourth follows from the fact that
\[
|\mathcal A\tau(\mathbf x)|\leq \|\mathcal A\|
\]
for any vector $\mathbf x:=(\mathbf x_1,\dots,\mathbf x_k)$ with $\|\mathbf x_i\|=1$ for all $i=1,\dots,k$ and the last one follows from $\alpha \leq\epsilon$. This, together with \eqref{eq:sub-block}, implies \eqref{eq:subdiff}.
\end{proof}
\subsection{Proof for Lemma~\ref{lem:gradient-diff}}\label{app:proofalgo-gradient}

\begin{proof}
{We let}
\[
W^{(i)}_{[p+1]}:=V^{(i)}\Lambda -V^{(i)}_{[p+1]}\Lambda^{(i)}_{[p+1]}-\alpha\big(U^{(i)}_{[p]}-U^{(i)}_{[p+1]}\big),
\]
where $\alpha\in\{0,\epsilon\}$ depending on whether there is a proximal correction or not (cf.\ the proof for Lemma~\ref{lem:subdiff}).
It follows from Lemma~\ref{lem:subdiff} that
\[
{\|W^{(i)}_{[p+1]}\|_F}\leq \gamma_0 \|U^{(i)}_{[p]}-U^{(i)}_{[p+1]}\|_F
\]
for some constant $\gamma_0>0$. By Algorithm~\ref{algo}, we have
\begin{equation}\label{eq:linear-polar}
V^{(i)}_{[p+1]}\Lambda^{(i)}_{[p+1]}+\alpha U^{(i)}_{[p]}=U^{(i)}_{[p+1]}S^{(i)}_{[p+1]}
\end{equation}
{where $S^{(i)}_{[p+1]}$ is a symmetric positive semidefinite matrix.} Since $U^{(i)}_{[p+1]}\in V(r,n_i)$ is an orthonormal matrix, we have
\begin{equation}\label{eq:linear-symmetry}
S^{(i)}_{[p+1]}=(U^{(i)}_{[p+1]})^\tp \big(V^{(i)}_{[p+1]}\Lambda^{(i)}_{[p+1]}+\alpha U^{(i)}_{[p]}\big)=\big(V^{(i)}_{[p+1]}\Lambda^{(i)}_{[p+1]}+\alpha U^{(i)}_{[p]}\big)^\tp U^{(i)}_{[p+1]},
\end{equation}
where the second equality follows from the symmetry of the matrix $S^{(i)}_{[p+1]}$.

Consequently, we have
\begin{align}
&\quad\ \frac{1}{2}\|\nabla_i f(\mathbb U_{[p+1]})-U^{(i)}_{[p+1]}(\nabla_i f(\mathbb U_{[p+1]}))^\tp U^{(i)}_{[p+1]}\|_F\nonumber\\
&=\|V^{(i)}\Lambda -U^{(i)}_{[p+1]}(V^{(i)}\Lambda)^\tp U^{(i)}_{[p+1]}\|_F\nonumber\\
&=\|W^{(i)}_{[p+1]}+V^{(i)}_{[p+1]}\Lambda^{(i)}_{[p+1]}+\alpha\big(U^{(i)}_{[p]}-U^{(i)}_{[p+1]}\big)-U^{(i)}_{[p+1]}(V^{(i)}\Lambda)^\tp U^{(i)}_{[p+1]}\|_F\nonumber\\
&\leq \|W^{(i)}_{[p+1]}\|_F+\alpha{\|U^{(i)}_{[p]}-U^{(i)}_{[p+1]}\|_F}+\|V^{(i)}_{[p+1]}\Lambda^{(i)}_{[p+1]}-U^{(i)}_{[p+1]}(V^{(i)}\Lambda)^\tp U^{(i)}_{[p+1]}\|_F\nonumber\\
&\leq \gamma_1\|U^{(i)}_{[p]}-U^{(i)}_{[p+1]}\|_F+\|V^{(i)}_{[p+1]}\Lambda^{(i)}_{[p+1]}-U^{(i)}_{[p+1]}(V^{(i)}\Lambda)^\tp U^{(i)}_{[p+1]}\|_F,\label{eq:relative-1}
\end{align}
{where $\gamma_1 = \gamma_0+\epsilon$}. {Next, we derive an estimation for the second summand of the right hand side of \eqref{eq:relative-1}. To do this, we notice that}
\begin{align}
&\quad\ \|V^{(i)}_{[p+1]}\Lambda^{(i)}_{[p+1]}-U^{(i)}_{[p+1]}(V^{(i)}\Lambda)^\tp U^{(i)}_{[p+1]}\|_F\nonumber\\
&=\|U^{(i)}_{[p+1]}S^{(i)}_{[p+1]}-\alpha U^{(i)}_{[p]}-U^{(i)}_{[p+1]}(V^{(i)}\Lambda)^\tp U^{(i)}_{[p+1]}\|_F\nonumber\\
&=\|U^{(i)}_{[p+1]}\big(V^{(i)}_{[p+1]}\Lambda^{(i)}_{[p+1]}+\alpha U^{(i)}_{[p]}\big)^\tp U^{(i)}_{[p+1]}-\alpha U^{(i)}_{[p]}-U^{(i)}_{[p+1]}(V^{(i)}\Lambda)^\tp U^{(i)}_{[p+1]}\|_F\nonumber\\
&\leq\|U^{(i)}_{[p+1]}\big((V^{(i)}_{[p+1]}\Lambda^{(i)}_{[p+1]})^\tp -(V^{(i)}\Lambda)^\tp \big)U^{(i)}_{[p+1]}\|_F+\alpha\| U^{(i)}_{[p+1]}(U^{(i)}_{[p]})^\tp U^{(i)}_{[p+1]}-U^{(i)}_{[p]}\|_F\nonumber\\
&\leq   \|V^{(i)}_{[p+1]}\Lambda^{(i)}_{[p+1]}-V^{(i)}\Lambda\|_F+\alpha\| U^{(i)}_{[p+1]}(U^{(i)}_{[p]})^\tp U^{(i)}_{[p+1]}-U^{(i)}_{[p]}\|_F\nonumber\\
&\leq \gamma_2\|U^{(i)}_{[p+1]}-U^{(i)}_{[p]}\|_F,\label{eq:relative-2}
\end{align}
where $\gamma_2>0$ is some constant, the first equality follows from \eqref{eq:linear-polar}, the second from \eqref{eq:linear-symmetry}, the second inequality\footnote{Since $U$ is orthonormal, we must have
\[
\|UAU\|_F^2=\|AU\|_F^2=\langle AUU^\tp,A\rangle\leq \|A\|_F^2.
\]
} follows from the fact that $U^{(i)}_{[p+1]}\in V(r,n)$ and {the last inequality from the relation
\[
\|(U^{(i)}_{[p]})^\tp U^{(i)}_{[p+1]}-I\|_F\leq \|U^{(i)}_{[p]}-U^{(i)}_{[p+1]}\|_F,
\]
which is obtained by Lemma~\ref{lem:distance}.
The desired inequality can be derived easily from \eqref{eq:relative-1} and \eqref{eq:relative-2}.}

\end{proof}
\end{document}